\documentclass[11pt]{article}
\usepackage[normalem]{ulem}
\usepackage{marginnote}
\usepackage{xcolor}
\usepackage{hyperref}
\usepackage{float} 
\usepackage{bbm}
\usepackage{amsmath,amsthm,amssymb,mathrsfs, enumerate}
\usepackage{graphicx,mdframed,booktabs}
\numberwithin{equation}{section}
\font\tencmmib=cmmib10 \skewchar\tencmmib '60
\newfam\cmmibfam
\textfont\cmmibfam=\tencmmib

\def\lessim{\ \lower4pt\hbox{$
		\buildrel{\displaystyle <}\over\sim$}\ }
\def\gessim{\ \lower4pt\hbox{$\buildrel{\displaystyle >}
		\over\sim$}\ }

\def\MS{{\text{\tiny MS}}}

\newcommand{\la}{\langle}
\newcommand{\ra}{\rangle}

\newcommand{\e}{\mathbb{E}}
\newcommand{\p}{\mathbb{P}}

\newcommand{\bsigma}{{\sigma}}
\newcommand{\NK}{N\!K}

\DeclareMathOperator*{\argmax}{arg max}

\newtheorem{lemma}{\bf Lemma}[section]

\newtheorem{theorem}{\bf Theorem}[section]
\newtheorem{corollary}{\bf Corollary}[section]
\newtheorem{remark}{\bf Remark}[section]

\newtheorem{proposition}{\bf Proposition}[section]

\newmdtheoremenv{theo}{Theorem}
\newcommand{\overlineit}[1]{\,\overline{\kern-0.2em #1\kern-0.1em }}

\newenvironment{Proof of lemma}{\noindent{\bf Proof of Lemma}}{\hfill$\Box$\newline}
\newenvironment{Proof of theorem}{\noindent{\it Proof of Theorem}}{\hfill\scriptsize{$\Box$}\newline}
\newenvironment{Proof of theorems}{\noindent{\bf Proof of Theorems}}{\hfill$\Box$\newline}
\newenvironment{Proof of proposition}{\noindent{\bf Proof of Proposition}}{\hfill$\Box$\newline}
\newenvironment{Proof of propositions}{\noindent{\bf Proof of Propositions}}{\hfill$\Box$\newline}
\newenvironment{Proof of exercise}{\noindent{\it Proof of Exercise:}}{\hfill$\Box$}
\begin{document}

\title{On the Fitness Landscape in the $\NK$ Model}	

 \author{Wei-Kuo Chen\thanks{University of Minnesota. Email: wkchen@umn.edu. Partly supported by NSF grant DMS-2246715 and Simons Foundation grant 1027727} \and  Si Tang \thanks{Lehigh University. Email: sit218@lehigh.edu. Partly supported by the NSF LEAPS-MPS Award DMS-2137614.} }

\maketitle 

\medskip
\begin{abstract}
The $\NK$ model, introduced by Kauffman, Levin, and Weinberger, is a random field used to describe the fitness landscape of certain species with $N$ genetic loci, each interacting with $K$ others. The model has wide applications in understanding evolutionary and natural selection as it captures ruggedness feature of the fitness landscape. Earlier literature has been focused on the case $K$ being a fixed positive integer and used tools from Ergodic and Markov theory. In this paper, by viewing it as a statistical physics object, we investigate the $\NK$ model in the regime $K/N\to\alpha \in(0,1]$ via the spin glass methodologies. Our main result identifies the exact limits for the free energy at any temperature and the maximum fitness. Moreover, we show that the $\NK$ model exhibits a multiple-peak structure, namely, the number of near-fittest genomes that are asymptotically orthogonal to each other is exponentially large. Based on establishing the overlap gap properties, we obtain quantitative descriptions for the geometry of the fitness landscape and deduce that,  in particular, near-fittest evolutionary paths become impossible as the fitness levels of the genomes approach the global maximum for any $\alpha\in (0,1]$. Nevertheless, we also show that by choosing $\alpha$ sufficiently small, an evolutionary path maintained at a given fitness level can be constructed with high probability.
\end{abstract}

\tableofcontents

\section{Introduction}\label{sec1}
The $\NK$ model is a random field that describes the fitness landscape of some species consisting of $N$ loci (or genes) in the genome, each interacting with $K$ other loci. It was first proposed by Kauffman, Levin and Weinberger \cite{KauffmanLevin198711, Kauffman1989, Weinberger89} in the late 80s, who studied the ruggedness of the fitness landscape of the $\NK$ model and the evolutionary dynamics on this landscape via adaptive walks. The model has drawn wide attentions of evolutionary biologists \cite{FL92,MP89, Stein1992}, as it captures the {\it epistatic} phenomenon of gene interactions \cite{Lush1935}, i.e., the effect on fitness at one locus also depends on the states of several other loci. In recent years, more mathematicians become interested in the model \cite{DL03, ES02,  LP04,WTZ2000}, as the fitness landscape of the $\NK$ model exhibits a tunable ruggedness \cite{Wright1932} as $K$ varies, which promotes the existence of exponentially many local attractors in the space of genomes, a crucial feature shared by many statistical physics models (e.g., \cite{BAFK21, Auffinger2013Complexity,auffinger2023complexity, panchenkobook, TopoComplexPRL13}) and verified by many experimental evolutionary processes in the lab (e.g., \cite{de2014empirical, papkou2023rugged, reetz2008constructing}).

In the classic setup of the $\NK$ model, a genome 
$\bsigma=(\sigma_0, \sigma_2, \ldots, \sigma_{N-1})$
is a binary sequence of length $N$, and at each genetic locus $i$,  the allele $\sigma_i$ may have two possible types, denoted as $0$ and $1$ (e.g., mutant and wildtype, or ressesive and dominant). For each genome $\bsigma\in \{0, 1\}^N$, its fitness $
H_{N,K}(\bsigma)$ is expressed as the sum of the $N$ fitness components $X_i$, contributed by each locus $i$. The gene interaction structure is incorporated into the fitness component $X_i$, which is determined by the $i$-th locus as well as $K$ other epistatic loci, that is
\begin{align}
\label{eqn:defHamiltonian}
H_{N,K}(\bsigma)&:=\sum_{i=0}^{N-1}X_i(\sigma_i,\sigma_{\!\!j^{(i)}_1},\ldots,\sigma_{\!j^{(i)}_K}).
\end{align}
Previous studies have concerned about different variations of the $\NK$ model. In \cite{Kauffman93, KauffmanLevin198711}, the $K$ epistatic loci were chosen randomly, whereas in \cite{DL03, ES02,  Kauffman1989,LP04,WTZ2000}, they were either the $K$ successors or the $K$ nearest loci. The state space may also be different, depending on the goal of the study. For example, besides the binary setting  \cite{ DL03, ES02, KauffmanLevin198711, LP04},  
$\bsigma$ may represent base pairs along a DNA sequence with each $ \sigma_i \in \{A, G, C, T\}$, the set of nucleotides, and this setup is suitable to study the fitness contribution of the single-nucleotide polymorphisms (SNPs) in the genome. In addition, 
when considering the evolution of protein molecules,  people often chose $\sigma_i \in \{\text{all amino acids}\}$ and view the functionality of proteins, such as structural stability and selectivity, as a measure of fitness. \cite{MP89,robustgeneticcode24,proteinevo-comp25}.

The parameter $K$ controls the degree of epistasis, which results in different levels of ruggedness of the fitness landscape in the $\NK$ model. One extreme case, $K = 0$, is trivial,
because in this case, genes simply do not interact with each other (i.e., there is no epistasis) and the fittest genome $\bsigma^\ast$ can be obtained by
maximizing each fitness component $X_i(\sigma_i)$ separately, i.e., $\sigma_i^\ast = \arg \max_{\sigma_i} X_i(\sigma_i)$ for each $i = 0, 1,\ldots, N-1$.
There is only one local maximal fitness value in this setting, and the landscape of the fitness values is smooth over the space of all genomes. For larger values of $K$,
the landscape becomes more rugged. The case when $K = N-1$ is the other extreme, in which fitness values of different genomes are mutually independent,and the model is
known as the random energy model (REM) \cite{Derrida-REM80, Derrida-REM81}. In this case, the expected number of local optima grows exponentially as $N$ increases; the largest fitness value is the maximum of $2^N$ i.i.d. random variables, which is well-understood using tools from extreme-value theory \cite{Resnick2007extreme}. Many properties of REM have been thoroughly explored in recent years --- see surveys \cite{Huang2021,MM2009-chp5} and references therein.

For other values of $K$, most previous results were based on numerical simulations, e.g., \cite{KauffmanLevin198711, robustgeneticcode24, proteinevo-comp25, fitness-evo02}, or some heuristic arguments \cite{Weinberger89}. Mathematically rigorous results are very limited, and we briefly summarize here. When $K$ is fixed, the asymptotic behavior of the global maximum was characterized independently in \cite{ES02} and \cite{DL03} using different approaches. In \cite{ES02}, Evans and Steinsaltz related the global maximum with the max-plus product of certain random matrices and identified the exact limit in the case when $K=1$ and the fitness components are independent standard exponential random variables. Similar studies were done using the substochastic Harris chains in \cite{DL03} for the negative exponential case with $K=1$.
For local maximum, \cite{ES02} and \cite{DL03} computed the expected fitness value, and in \cite{DL03} a central limit theorem and a large deviation result were also obtained for the negative exponential case with $K=1$. Both studies proved an exponential growth for the number of local maxima as $N\to \infty$ and identified the exponents in the aforementioned special cases. The methods in \cite{DL03,ES02} apply to other values of $K$ provided $K$ stays fixed as $N\to\infty$. However, analytical computations for the cases when $K>1$ or when the fitness components follow other distributions are quite involved. The other scheme in which $K$ could grow with $N$ is also biologically relevant as indicated in \cite{Kauffman93}. This case was first studied heuristically by Weinberger in \cite{Weinberger89}, and some of the claims were later made rigorous in \cite{LP04}. Among other things, Limic and Pemantle showed that when $K/N\to \alpha$ as $N\to \infty$, the probability of a given genome being a local maxima behaved like $N^{-\alpha}$ when the fitness components are independent Gaussians. They also conjectured this would hold universally for other fitness component distributions and proved partial results that support their conjecture.

Despite of the results mentioned above, many fundamental questions are still unaddressed, especially concerning the fitness landscape in the regime when $K/N\to \alpha$ for some fixed constant $\alpha\in (0,1)$.  A key direction is to study the precise limiting behavior of the maximum fitness, as well as the geometric structure of its level sets.  
From the perspective of evolutionary dynamics, it is also of great interest to understand how one fittest genome could evolve into another. In this paper, we pursue these directions within the $\NK$ model with adjacent neighborhoods in the limit $K/N \to \alpha$, employing tools from mathematical spin glass theory \cite{panchenko2014parisi, talagrand2010mean, talagrand2011mean}. We investigate its statistical physics properties such as the asymptotic behavior of the overlap between spin configurations (i.e., genomes), the Gibbs measure, and the free energy. Ultimately, we aim to use these results to address the questions posed above and offer insights into the underlying evolutionary dynamics.

\subsection{Model Setup and Main Results}\label{sec:mainresults}

For each locus $i$, we consider its epistatic loci to be its $K$ successors, i.e., $(j^{(i)}_1,\ldots,j^{(i)}_K)=(i+1, i+2, \ldots, i+K)$ in \eqref{eqn:defHamiltonian}, with $K$ chosen to be asymptotically proportional to $N$ as $N\to \infty$. Specifically, for every $N\geq 1,$ let $K=\lfloor \alpha (N-1)\rfloor$ for some $\alpha 
\in (0, 1]$ fixed. The boundaries are identified, namely, $\sigma_{i+N} =\sigma_{i}$ for each $i$. The alleles are binary, however, we set the the two possible alleles to be $+1$ and $-1$, adopting the convention in statistical physics studies. We also denote $\Sigma_n := \{+1, -1\}^n$ for any integer $n\ge 1$. For each $i=0,1, \ldots, N-1$, the fitness components $\{X_i(\sigma_i, \sigma_{i+1}, \ldots, \sigma_{i+K}
): (\sigma_i,\ldots, \sigma_{i+K})\in \Sigma_{K+1} \}$ at the $i$-th locus are i.i.d. standard Gaussian random variables; these fitness components are also assumed to be independent of each other. 

With this setup,  $(H_{N,K}(\bsigma))_{\bsigma\in \Sigma_N}$ may be viewed as a centered Gaussian process over $\Sigma_N$, whose  covariance structure is given by 
\begin{align*}
    \e H_{N,K}(\bsigma^1)H_{N,K}(\bsigma^2)&=NQ(\bsigma^1,\bsigma^2), \quad \forall\bsigma^1, \bsigma^2\in \Sigma_N,
\end{align*}
where $Q(\bsigma^1,\bsigma^2)$ is called the {\it epistatic overlap} between the two genomes $\bsigma^1$ and $\bsigma^2$, defined as
$$
Q(\bsigma^1,\bsigma^2):=\frac{1}{N}\sum_{i=0}^{N-1}I(\sigma_i^1=\sigma_i^2,\ldots,\sigma_{i+K}^1=\sigma_{i+K}^2).
$$
We remark that, $Q$ is invariant to the exactly value (no matter $\{0,1\}$, $\{+1, -1\}$, or some other binary set) that we assign to the alleles and only depends on the agreement of the alleles at each locus. We refer to $Q$ as the epistatic overlap to distinguish it from the usual {\it overlap}, a fundamental quantity considered in many spin glass models and is defined as the (normalized) scalar product,
 \begin{align*}
     R(\bsigma^1,\bsigma^2)=\frac{1}{N}\sum_{i=0}^{N-1}\sigma_i^1\sigma_i^2.
 \end{align*}
This is another quantity that measures how similar the two genomes are. For any two genomes $\sigma^l, \sigma^{l'} \in \Sigma_N$, we will abbreviate their overlaps as $Q_{l,l'}:=Q(\sigma^l, \sigma^{l'})$ and $R_{l,l'}:=R(\sigma^l, \sigma^{l'})$, when there is no ambiguity.
\begin{remark}\rm
  Note that the overlap $R$ is asymptotically stable 
as long as the signs of no more than $o(N)$ loci are flipped, 
however, this is not the case for the epistatic overlap $Q$. 
To see this, for example, for any $\sigma^1$, letting $\sigma^2$ be constructed from $\sigma^1$ by flipping the sign of the first locus, $\sigma^2_0=-\sigma^1_0$, we have $R_{1,2}=(N-1)/N\approx R_{1,1}$, but $Q_{1,2}=(N-K-1)/N\approx1-\alpha < Q_{1,1}$. 
\end{remark}

 In this paper, we view the $\NK$ model as a statistical physics model by interpreting the genome $\sigma$ and the fitness function $H_{N,K}$ as the spin configuration and the Hamiltonian (or energy), respectively. The maximal fitness and the fittest genome, namely,
\begin{align}
	\label{eqn:def-MN}
	M_{N,K}:=\max_{\bsigma\in \Sigma_N}\frac{H_{N,K}(\bsigma)}{N}\ \ \mbox{and}\ \ \sigma^*:=\argmax_{\sigma\in \Sigma_N}\frac{H_{N,K}(\sigma)}{N},
\end{align}
are correspondingly called the ground state energy and the ground state. The fundamental idea of statistical physics for handling these objectives is to introduce the Gibbs measure, which is defined as, for a given (inverse) temperature $\beta>0$, 
\[
G_{N,K,\beta}(\bsigma)= \frac{e^{\beta H_{N,K}(\bsigma)}}{Z_{N,K}(\beta)},\ \ \forall \bsigma\in \Sigma_N,
\]
where $Z_{N,K}(\beta):=\sum_{\sigma}e^{\beta H_{N,K}(\sigma)}$ is called the partition function. The Gibbs measure describes the chance that a given genome $\sigma \in \Sigma_N$ is observed for a given parameter $\beta >0$, 
which can be interpreted as a  measure of natural selection in evolution: when $\beta$ is large, the Gibbs measure is highly concentrated on genomes with higher fitness values and natural selection would favor much strongly toward these genomes; while if $\beta$ is small, all genomes would have comparable chances to be selected by evolution. Denote by $\bsigma^1,\bsigma^2,\ldots$ a collection of independently sampled genomes (called the {\it replicas}) from the Gibbs measure $G_{N,K,\beta}$ and by $\la \cdot\ra_{\beta}$ the expectation with respect to $G_{N,K,\beta}$. The free energy is defined as
\begin{align*}
	F_{N,K}(\beta)&=\frac{1}{N}\ln Z_{N,K}(\beta),
\end{align*}
which plays the role as a soft approximation for $M_{N,K}$ through the simple bounds,
\begin{equation}
	\label{eqn:GSEbounds}
	M_{N,K}\leq \frac{F_{N,K}(\beta)}{\beta} \leq \frac{\ln 2}{\beta}+M_{N,K}.
\end{equation} 

\subsubsection{Free Energy, Maximum Fitness, and Multiple Peaks}
Our first main result characterizes the limiting behavior of the free energy and the maximum fitness as $N\to \infty$. While they appear to be not depending on the choice of $\alpha\in (0, 1]$, the free energy exhibits a phase transition at the critical temperature $\beta_c:=\sqrt{2\ln2} \approx 1.1774$. 
\begin{theorem}\label{thm:Fnk-Mnk}
	For any $\alpha\in (0,1],$ we have that 
	\begin{align*}
		\lim_{N\to\infty}\e F_{N,K}(\beta)&=\left\{\begin{array}{ll}
			\ln 2+\frac{\beta^2}{2},&\forall \beta<\beta_c,\\
			\\
			\beta \beta_c,&\forall \beta \geq \beta_c.
		\end{array}\right.
	\end{align*}
	It follows from \eqref{eqn:GSEbounds} that
	$\displaystyle\lim_{N\to\infty}\e M_{N,K}=\beta_c$.
\end{theorem}
As mentioned before, when $K=N-1$ (i.e., $\alpha=1$), the $NK$ model is exactly the same as  REM. Our results in Theorem \ref{thm:Fnk-Mnk} show that the limiting free energy and maximum fitness of the $\NK$ model agree with those of REM, and this holds for all values of $\alpha \in (0,1]$. The next result characterizes a multiple-peak property for the near-fittest genomes in the $\NK$ model.

\begin{theorem}\label{thm:peaks}
	Let $\alpha\in(0,1].$ For any $0<\varepsilon<\beta_c,$ there exist two absolute constants $C_1,C_2>0$ independent of $N$ such that with probability at least $1-C_1e^{-N/C_1}$, there exists a set $S_N(\varepsilon)\subset \Sigma_{N}$ with the following properties
	\begin{enumerate}
		\item $|S_N(\varepsilon)|\geq e^{C_2N}$;
		\item $\frac{H_{N,K}(\sigma)}{N}\geq M_{N,K}-\varepsilon$ for all $\sigma\in S_N(\varepsilon)$;
		\item $Q(\sigma,\sigma')=0$ and $|R(\sigma,\sigma')|<\varepsilon$ and  for any two distinct $\sigma,\sigma'\in S_N(\varepsilon).$
	\end{enumerate}
\end{theorem}

\begin{remark}\rm Theorem \ref{thm:peaks} states that for any value $\alpha\in (0,1]$, the landscape of the $\NK$ model contains exponentially many near-fittest genomes that are approximately pairwise orthogonal (with respect to either $Q$ or $R$).
	Note that, with $R_{1,2}\approx 0$, any near-fittest genome pairs $\sigma^1 \ne \sigma^2$ from $S_N(\varepsilon)$ would disagree at approximately $N/2$ loci. Furthermore, since $Q_{1,2}=0$, these mismatched loci can not be too clustered together in the genome such that a perfect match of $K+1$ (or more) consecutive loci may still exist.
\end{remark}
\begin{remark}\label{rmk1.3}\rm
	In an earlier work \cite{chatterjee2008chaos}, Chatterjee showed that there exists a constant $C>0$ such that for any $N>K\geq 1$, with probability at least $1-\delta$ for $\delta:=C(\ln K)^{-1/2}$, there is a subset $A\subset \Sigma_N$ satisfying that $|A|\geq 1/\delta,$ $H_{N,K}(\sigma)/N\geq (1-\delta)M_{N,K}$ for all $\sigma\in A,$ and $Q(\sigma,\sigma')\leq \delta$ for all distinct $\sigma,\sigma'\in A$. Under our setting $K\approx \alpha N$, Chatterjee's result suggests that the $|A|$ is at least of logarithmic scale in $N$. Our Theorem \ref{thm:peaks} improves this estimate by establishing an exponential growth for the size of the set, and moreover, the orthogonality of these genomes also holds under the metric $R(\sigma,\sigma')$. The statements can be further strengthened in the regime $\alpha > \alpha_\ast$ (see Corollary \ref{cor1.1} later), where the set $S_N(\varepsilon)$ and the constant $C_2$ can be described explicitly. 
\end{remark}
\begin{remark}\rm
	Previous work \cite{DL03,ES02, LP04,Weinberger91} have concerned about the number of local fitness maxima of the $\NK$ model for various distributions of the fitness components and for $K$ being either fixed or growing linearly with $N$. Note a genome $\sigma$ is called an local fitness maximum if $H_{N,K}(\sigma)\ge H_{N,K}(\sigma')$ for any $\sigma'$ that differs from $\sigma$ at exactly one locus. In these studies, local fitness maxima are found to be exponentially many --- however, this does not imply that the number of near-fittest genomes behaves the same way. In fact, in a recent large-scale evolutionary experiment \cite{papkou2023rugged}, it was found that most local maxima had low fitness values; there were only 74 high-fitness peaks out of the 514 local maxima identified in the experiment. Our result in Theorem \ref{thm:peaks} confirms that the number of near-fittest genomes also grows exponentially with $N$ as $N\to\infty$.
\end{remark}

\subsubsection{Overlap Gap Properties}
Following the multiple-peak property, several question naturally arise. For instance, {\it are {\bf all} near-fittest genomes pairwise orthogonal, as those in the set $S_N(\epsilon)$}? {\it How are near-fittest genomes interconnected?} In particular, as proposed by S. Evans and discussed in \cite[Open Problem 4.5]{chatterjee2014superconcentration}, {\it 
	is it possible to evolve from one near-fittest genome to another along a path where all intermediate genomes remain near-fittest?}
More precisely, one is interested in evolutionary paths connecting two near-fittest genomes of $O(1)$ many steps such that (i) in each step only a small fraction of the $N$ loci mutate, and (ii) all intermediate genomes along the path remain near-fittest. 
These requirements are especially relevant in the context of evolution: firstly, for natural populations, it is unlikely to have too many mutations in one generation; secondly, genomes of lower fitness are more likely to be eliminated under natural selection, making evolutionary paths through such states significantly less probable.  

We address the questions mentioned above by characterizing the fitness landscape in terms of the overlaps $Q$ and $R$. Let \begin{align}\label{def:alpha-star}
	\alpha_\ast:=3-2\sqrt{2}\approx 0.172.
\end{align}
We arrange our results in this subsection according to the value of the epistasis parameter $\alpha$, starting from $\alpha>\alpha_*$, the high-epistasis regime. 
Our first result below shows that the epistasis overlap $Q_{1,2}$ of two independently sample genomes undergoes a phase transition at $\beta_c$. 

\begin{theorem}\label{thm:Q-high-epi}
	Assume $\alpha_*<\alpha\le 1$ and recall that $\beta_c = \sqrt{2\ln 2}$. For any $\beta\leq \beta_c$,
	\begin{align}\label{eqn:Q-high-epi-hightemp}
		\lim_{N\to\infty}\e \bigl\la I(Q_{1,2}=0)\bigr\ra_\beta=1,
	\end{align}
	and for any $\beta>\beta_c$,
	\begin{align}
		\lim_{N\to\infty}\e\bigl\la I(Q_{1,2}=0)\bigr\ra_\beta&= \frac{\beta_c}{\beta},\label{eqn:Q-high-epi-lowtemp-1}\\
		\limsup_{N\to\infty}\e\bigl\la I(Q_{1,2}=1)\bigr\ra_{\beta}&=1-\frac{\beta_c}{\beta}.\label{eqn:Q-high-epi-lowtemp-2}
	\end{align}
\end{theorem}

This theorem states that  at weaker strength of  natural selection ($\beta < \beta_c$), two independently sampled genomes from  the Gibbs measure are always {\it completely decoupled}, i.e., a mismatch exists in any segment of length $K+1$; whereas at higher strength of natural selection ($\beta > \beta_c$), they will either align perfectly with each other or are completely decoupled.

Next, we establishes the so-called {\it overlap gap property} in the sense that the overlaps between genomes whose fitness exceed a certain level are forbidden to charge values in some sub-intervals. To this end, for any nonempty set $S\subset \Sigma_N\times\Sigma_N$, we define the coupled maximum fitness constrained by $S$ as
\begin{align}\label{def:barM}
	\overlineit{M}_{N,K}(S)& : =\max_{(\sigma^1,\sigma^2)\in S}\Bigl(\frac{H_{N,K}(\sigma^1)}{N}+\frac{H_{N,K}(\sigma^2)}{N}\Bigr).
\end{align}
The following result provides quantitative bounds for the level of fitness and the overlaps of the genomes in the high-epistasis regime.

\begin{theorem}\label{thm:nearfittest-highepi}
	Assume $\alpha_*<\alpha\leq 1.$
	The follow statements hold. 
	\begin{enumerate}
		\item We have
		\begin{align}
			\begin{split}\label{eqn:nearfittest-highepi-Q}&\limsup_{N\to\infty}\e\overlineit{M}_{N,K}(\{0<Q_{1,2}<1\})\\
				&\leq 2\lim_{N\to\infty}\e M_{N,K}-2\beta_c\Bigl(1-\frac{3-\alpha}{2\sqrt{2}}\Bigr).
			\end{split}
		\end{align}
		
		\item For any $\delta\in (0,1),$
		\begin{align}
			\begin{split}\label{eqn:nearfittest-highepi-R}&\limsup_{N\to\infty}\e\overlineit{M}_{N,K}(\{\delta<|R_{1,2}|<1\})\\
				&\leq 2\lim_{N\to\infty}\e M_{N,K}-2\beta_c\min\Bigl(1-\sqrt{\frac{h(\delta)}{2}},1-\frac{3-\alpha}{2\sqrt{2}}\Bigr),
			\end{split}
		\end{align}
		where for $u\in [0,1],$
		\begin{align}\label{eqn:def-h}
			h(u)&:=1-\frac{1+u}{2}\log_2\frac{1+u}{2}-\frac{1-u}{2}\log_2\frac{1-u}{2}
		\end{align}
		is a strictly decreasing function with $h(0)=2$ and $h(1)=1.$
		
	\end{enumerate}

\end{theorem}

Note that the maximum fitness $M_{N,K}$ and $\overlineit{M}_{N,K}(S)$ here are all concentrated with high probability (see Lemma \ref{lem3} below) so that the above results are essentially valid without taking the expectations, where an additional error on which can be made as small as possible will be introduced to the above inequalities. 
For the epistasis overlap $Q$, noting that $\lim_{N\to\infty}\e M_{N,K}=\beta_c,$ we readily see from \eqref{eqn:nearfittest-highepi-Q} that with high probability, if $\sigma^1 \ne \sigma^2$ are two genomes with fitness at least $N\beta_c(3-\alpha)/(2\sqrt{2}),$ then they must be completely decoupled, i.e.,  $Q_{1,2}=0$. The inequality \eqref{eqn:nearfittest-highepi-R} quantifies how the overlap $R_{1,2}$ changes with the fitness level on the landscape: for any $\delta\in (0,1)$, if $\sigma^1 \ne \sigma^2$ are two genomes with fitness at least $NE(\delta)$, where
	\begin{align}\label{energylevel:eq1}
		E(\delta):=\lim_{N\to\infty}\e M_{N,K}-\beta_c\min\Bigl(1-\sqrt{\frac{h(\delta)}{2}},1-\frac{3-\alpha}{2\sqrt{2}}\Bigr),
	\end{align}
	then their overlap must satisfy $|R_{1,2}|\leq \delta$, and thus, their distance is in between $\sqrt{2(1-\delta)}$ and $\sqrt{2(1+ \delta)}$. 
    Consequently, this 
    implies that whenever $\eta,\delta\in(0,1)$ with $\eta<(1-\delta)/2$, it is not possible to exist an evolutionary path, along which the genome in each generation maintains fitness level at last $E(\delta)$, in the meantime, any two consecutive generations have no more than $\eta N$ mutations since, in this case, their overlaps will exceed $1-2\eta>\delta .$ In particular,
    as $\delta\to0$, $E(\delta)$ becomes closer to the optimal fitness and an evolution path consisting of near-fittest genomes must involve more than $\eta N\gtrapprox N/2$ mutations in at least one generation, which is biologically improbable.
     See Figure \ref{figure1}(a) for the structure of the near-fittest genomes when $\delta\approx 0$ and Figure \ref{figure-E-delta} (Left) for how the fitness level $E(\delta)$  changes with $\delta$ for a few $\alpha$ values in the high-epistasis regime $(\alpha_\ast, 1]$. 
     
\begin{figure}[ht]
    \centering
	\includegraphics[width=0.90\textwidth]{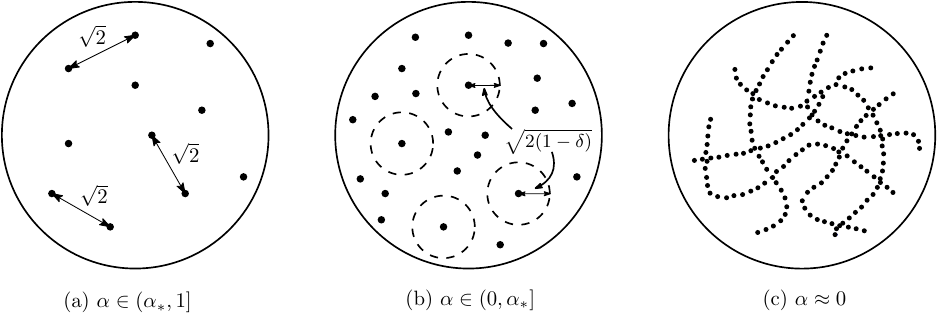}
	\caption{Schematic figures of the near-fittest genomes with respect to the metric $R$. In (a), the genomes are approximately mutually orthogonal so that their distances are almost $\sqrt{2}.$ In (b), the genomes all keep at least a fixed positive distance, $\sqrt{2(1-\delta)}$, from each other for some $\delta\in (0,1).$ In (c), the genomes asymptotically commute between each other.}
	\label{figure1}
\end{figure}

As a corollary, when $\alpha>\alpha_*$, we get a stronger result than Theorem \ref{thm:peaks}, for which the set $S_N(\varepsilon)$ can be taken as the level set, $
\mathcal{L}_N(\varepsilon):=\{\sigma:H_{N,K}(\sigma)/N\geq M_{N,K}-\varepsilon\}$
for $\varepsilon>0$, and its cardinality can be described more explicitly.

\begin{corollary}\label{cor1.1}
	Assume $\alpha_*<\alpha\leq 1$. For any $0<\delta<1$, consider the level set $\mathcal{L}_N(\varepsilon)$ with
	\begin{align}\label{cor1.1:eq1}
		\varepsilon=\frac{\beta_c}{7}\min\Bigl(1-\sqrt{\frac{h(\delta)}{2}},1-\frac{3-\alpha}{2\sqrt{2}}\Bigr).
	\end{align}
	Then, for any $\gamma>0$, with probability at least  $1-Ce^{-N/C}$ for some $C=C(\delta, \gamma)$, $\mathcal{L}_N(\varepsilon)$ satisfies the following 
	\begin{enumerate}
		\item for any two distinct $\sigma,\sigma'\in \mathcal{L}_N(\varepsilon),$ $Q(\sigma,\sigma')=0$ and $|R(\sigma,\sigma')|<\delta$; 
		\item $\Bigl|\frac{1}{N}\ln |\mathcal{L}_N(\varepsilon)|-\varepsilon\bigl(\beta_c-\frac{\varepsilon}{2}\bigr)\Bigr|<\gamma$.
	\end{enumerate}
\end{corollary}

\begin{figure}[H]
    \centering
	\includegraphics[width=0.4\textwidth]{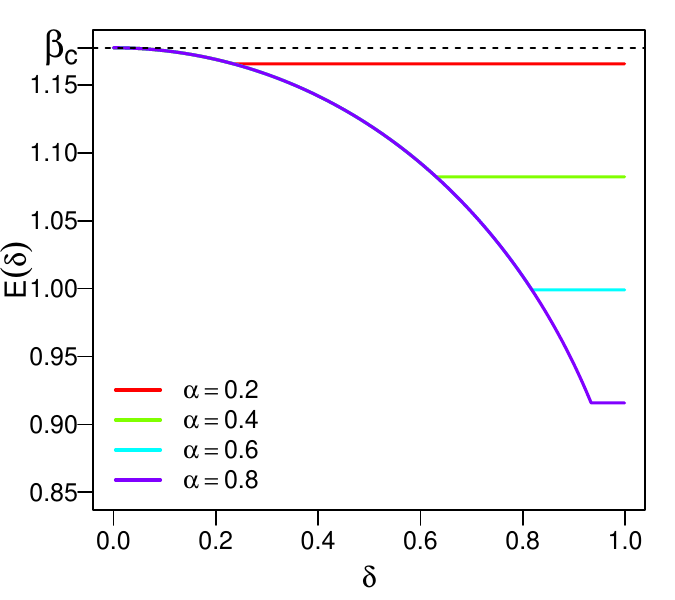}
    \includegraphics[width=0.4\textwidth]{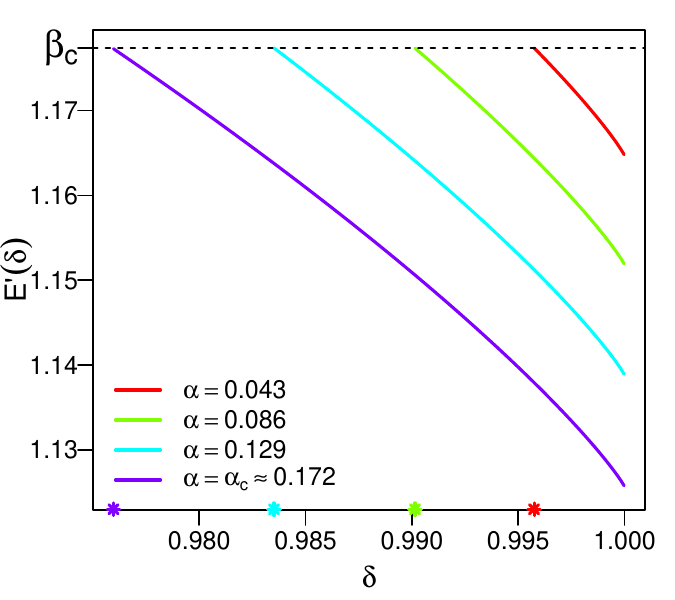}
     \caption{(Left) A sketch of $E(\delta)$ for $0<\delta<1$ when $\alpha>\alpha_*$. As $\delta$ increases the value of $E(\delta)$ will first decrease along the purple curve and then move horizontally to the right depending on the epistasis parameter. (Right) A sketch of $E'(\delta)$ when $0<\alpha\leq \alpha_*$. Note that $E'(\delta)$ is defined only on $(\delta_*,1)$ for $\delta_*=h^{-1}(2/(2-\alpha))$, which is marked as colored points on the $x$-axis for the corresponding $\alpha$ values.}
      \label{figure-E-delta}
\end{figure}

Our next theorem proves analogous results of Theorems \ref{thm:Q-high-epi} and \ref{thm:nearfittest-highepi} in the low-epistasis regime $\alpha\in (0,\alpha_*]$. To state the result, we first define \begin{align*}
	c_1(\alpha)&:=\frac{1}{2}\bigl(1-\alpha-\Delta(\alpha)\bigr)>0,\\
	c_2(\alpha)&:=\frac{1}{2}\bigl(1-\alpha+\Delta(\alpha)\bigr)<1-\alpha,
\end{align*}
where $\Delta(\alpha)=\sqrt{\alpha^2-6\alpha+1}$.

\begin{theorem}\label{thm:near-fittest-lowepi}
	Assume $0<\alpha\le\alpha_*$. If $\beta<\beta_c,$ \eqref{eqn:Q-high-epi-hightemp} holds
	and if $\beta>\beta_c,$ we have that for any $0<\delta<c_1(\alpha),$
	\begin{align}\label{thme.1:eq0} 
		\limsup_{N\to\infty}\e\bigl\la I(Q_{1,2}\in (0, c_1(\alpha)-\delta] \cup[c_2(\alpha)+\delta,1)\bigr\ra_\beta=0.
	\end{align}
	Additionally, we have the following statements.
	\begin{enumerate}
		\item For any $0<\delta<c_1(\alpha),$ we have
		\begin{align}
			\begin{split}
				\label{thm:low-epi:eq1} 
				&\limsup_{N\to\infty}\e\overlineit{M}_{N,K}(\{Q_{1,2} \in (0, c_1(\alpha)-\delta] \cup[c_2(\alpha)+\delta,1)\})\\
				&\leq 2\lim_{N\to\infty}\e M_{N,K}-2\beta_c\Bigl(1-\sqrt{1-\frac{\delta(\Delta(\alpha)+\delta)}{2}}\Bigr).
			\end{split}
		\end{align}

		\item For any $0<\delta<1$ with $h(\delta)<2/(2-\alpha)$, we have
		\begin{align}
			\begin{split}\label{thm3:eq3}
    &\limsup_{N\to\infty}\e\overlineit{M}_{N,K}(\{\delta<|R_{1,2}|<1\})\\
				&\quad \leq 2\lim_{N\to\infty}\e M_{N,K}-2\beta_c\Bigl(1-\sqrt{\frac{2-\alpha}{2}h(\delta)}\Bigr).
			\end{split}
		\end{align}
	\end{enumerate}
\end{theorem}

This result shows that when $\alpha<\alpha_*$, the epistatic overlap $Q$ can only be supported on the set $\{0\}\cup[c_1(\alpha),c_2(\alpha)]\cup\{1\}$ when $\beta>\beta_c$ and the same holds for the near-fittest genomes. 
Analogously to \eqref{eqn:nearfittest-highepi-Q}, \eqref{thm3:eq3} provides a quantitative characterization of the overlaps for near-fittest genomes in the low-epistasis regime: for any $\delta>\delta_*$ with  $\delta_\ast = \delta_\ast(\alpha)= h^{-1}(2/(2-\alpha))$, if the fitness values of two genomes $\sigma^1\ne \sigma^2$ are at least 
    \begin{align}\label{energylevel:eq2}
    E'(\delta):=\lim_{N\to\infty}\e M_{N,K}-\beta_c\Bigl(1-\sqrt{\frac{2-\alpha}{2}h(\delta)}\Bigr),
    \end{align}
	then their overlap $|R_{1,2}|<\delta$. The structure of the near-fittest genomes and the behavior of $E'(\delta)$ in the $\alpha<\alpha_*$ regime is sketched in Figure \ref{figure1}(b) and Figure \ref{figure-E-delta} (Right), respectively. Similar to the high-epistasis regime, we see that as $\delta \to \delta_\ast$, $E'(\delta)$ approaches the optimal fitness and thus,  there are no near-fittest evolutionary paths that involve no more than $\eta N$ mutations at every step whenever $0<\eta<(1-\delta_*)/2$.

\begin{remark}
	\rm Overlap gap properties have also been found in many important spin glass models (see, e.g., \cite{auffinger2018energy, chen2019suboptimality, chen2018energy, chen2018disorder}). Similar to the evolutionary barrier we explained above, it was shown to be a fundamental factor in determining the computational hardness of various related optimization problems (see, e.g., \cite{gamarnik2021overlap, huang2025tight}). 
\end{remark}

\subsubsection{Near-Fittest Evolutionary Paths}

Note that in the low-epistasis regime, $\delta_\ast\to 1$ as $\alpha\to 0$. Theorem \ref{thm:near-fittest-lowepi} suggests that  it is possible to have two genomes with fitness level greater than $E'(\delta)$ with an overlap in the range $(0, \delta)$ when $\delta>\delta_\ast$, which is almost the entire interval $(0,1)$ as long as $\alpha$ is small enough. The existence of the evolutionary paths would, in particular, rely on whether the overlap can take values as close as possible to $1$. 
Our final result shows that this is indeed the case and provides an affirmative answer to Evan's question.

Let $n\ge 10$ be an integer and $k=\lfloor N/(n+1)\rfloor$. Consider any two genomes $\hat\sigma \ne \check \sigma$ and an evolutionary path of $n$ steps
\begin{align}
\label{path}
\sigma^{(0)}=\hat \sigma,\, \sigma^{(1)},\, \sigma^{(2)},\, \ldots, \sigma^{(n-1)}, \,\sigma^{(n)}=\check\sigma, 
\end{align}
where for $1\leq l\leq n-1,$
\begin{align}\label{def:bridge}
\sigma_i^{(l)}=\left\{
\begin{array}{ll}
\hat \sigma_i & \text{if } i \ge lk \\
\\
\check \sigma_i & \text{if } i < lk,
\end{array}
\right.
\end{align}
that is, for the steps $l=0, 1,\ldots, n-2,$ $\sigma^{(l)}\to \sigma^{(l+1)}$, the genes on the set of loci $I_l=\{lk, lk+1, \ldots, (l+1)k-1\}$ are updated to be the same genotypes as the corresponding loci in $\check \sigma$, and in the last step from $\sigma^{(n-1)}$ to $\sigma^{(n)},$ the updates occur only at the set of loci $I_{n-1}=\{(n-1)k,(n-1)k+1,\ldots,N-1\}$. Our next result shows that if $\hat \sigma$ and $\check \sigma$ are near-fittest, then this evolutionary path \eqref{path} is near-fittest for sufficiently small $\alpha$.

\begin{theorem}\label{thm:near-fittest-path} 
Let $n\geq 10$ and $\eta\in (0,1)$. For any $\alpha$ with
\begin{align}\label{bridge:eq14}
    0<\alpha<\frac{\eta}{5\sqrt{2\ln 2}},
\end{align}
there exists some $N_0$ depending on $\eta,\alpha$ such that for any $N\geq N_0,$ with probability at least $1-\omega e^{-\eta^2N/(n^2\omega)}$, whenever $\check\sigma$ and $\hat\sigma$ satisfy
\begin{align}\label{bridge:eq13}
		\min\Bigl(\frac{H_{N,K}(\hat \sigma)}{N},\frac{H_{N,K}(\check \sigma)}{N}\Bigr)\geq M_{N,K}-\eta,
	\end{align}
we have that
	\begin{align}
		\nonumber   \min_{0\leq l\leq n-1} Q(\sigma^{(l)},\sigma^{(l+1)})&\geq 1-C\Bigl(\frac{1}{n}+\eta\Bigr),\\
		 \min_{0\leq l\leq n-1} R(\sigma^{(l)},\sigma^{(l+1)})&\geq 1-\frac{C}{n},\\
        \label{thm:near-fittest-path:eq11} 
        \min_{1\leq l\leq n-1}\frac{H_{N,K}(\sigma^{(l)})}{N}&\geq M_{N,K}-Cn\eta.
	\end{align}
    Here, $\omega>0$ and $C>0$ are absolute constants independent of all other variables.
\end{theorem}

 We emphasize that under \eqref{bridge:eq14}, Theorem \ref{thm:peaks} still ensures the existence of exponentially many near-fittest genomes that are approximately orthogonal to each other since $\alpha$ is strictly positive. Now for any fixed number of steps $n,$ as $\eta$ is allowed to be as small as one wishes, Theorem~\ref{thm:near-fittest-path} further implies that all near-fittest genomes are essentially interconnected by the near-fittest evolutionary paths \eqref{path}, in which each generation involves only no more than $\lfloor N/n\rfloor$ mutations; see Figure \ref{figure1}(c). We add that Theorem \ref{thm:near-fittest-path} indeed also holds disregard the ordering of updates on the sets of loci, $I_1,\ldots,I_{n-1}$.

\begin{remark}
    \rm As long as $n$ is large enough, all overlap values become achievable along the path \eqref{path} in the additional limit as $\alpha \to 0$, that is, for any $t\in [0,1],$
\begin{align*}
	\limsup_{\varepsilon\downarrow 0}\limsup_{\alpha\downarrow 0}\limsup_{N\to\infty}\e\overlineit{M}_{N,K}(\{|Q_{1,2}-t|<\varepsilon\})=2\lim_{N\to\infty} \e M_{N,K},\\
	\limsup_{\varepsilon\downarrow 0}\limsup_{\alpha\downarrow 0}\limsup_{N\to\infty}\e\overlineit{M}_{N,K}(\{|R_{1,2}-t|<\varepsilon\})=2\lim_{N\to\infty} \e M_{N,K}.
\end{align*}
This behavior is in contrast to the overlap gap property described earlier for $\alpha\in(0,1]$ fixed.
\end{remark}

\subsection{Open Problems}

While we investigate the $\NK$ model in the regime $K/N\to\alpha \in (0,1]$ using methodologies from spin glass theory, several important problems remain open both within this regime and beyond. Below, we outline some of these open problems.

\begin{enumerate}
	
	\item Let $K$ be a fixed positive integer. What is the limiting free energy? What is the behavior of the overlap at any temperature and at the maximal fitness? Notably in this regime, the maximum fitness has been well-studied by the Furstenberg–Kifer theory, see \cite{ES02}. We anticipate that similar approach can be used to understand the limiting free energy, but the study for the overlaps requires  new ideas.
	
	\item Understand the questions above under the assumption, $K\to\infty$ and $K/N\to 0$ as $N\to\infty$.  Does Chatterjee's result \cite{chatterjee2008chaos} elaborated in Remark \ref{rmk1.3} remain true for overlap $R(\sigma,\sigma')$ instead? If yes, we will see that the number of near-fittest and orthogonal genomes diverges. In this case, does the near-fittest path in Theorem \ref{thm:near-fittest-path} still exist? 
	
	\item Suppose $K/N\to\alpha\in (0,1].$ Our formulation of the $\NK$ model assumes that the fitness components are normal. Can we extend the results to other distributions?
\end{enumerate}

\subsection{Structure of the Paper}

Theorem \ref{thm:Fnk-Mnk} is the most crucial ingredient throughout this paper, providing the basis for all subsequent results.
As mentioned before, the $\NK$ model with $\alpha=1$ corresponds to the REM and it was well-known that the second moment method allows to obtain the same formulas in Theorem \ref{thm:Fnk-Mnk}, see, e.g., \cite{bovier2006statistical}. In light of this, it is tempting to apply the same approach to the $\NK$ model. However, while the details will be carried out in Section \ref{sec:mmomentmethod}, our analysis indicates that the moment method seems to work only in the regime $[\alpha_*,1]$, see Remark \ref{remark:momentmethod} below. In Section \ref{sec6}, we establish Theorem \ref{thm:Fnk-Mnk} for the entire regime $\alpha\in (0,1]$ by an approximation argument that makes use of a $p$-spin variant of the $\NK$ model, called the $p$-spin $\NK$ model, and a balanced multi-species model. The proofs of the overlap gap properties stated in Theorem \ref{thm:Q-high-epi}, Corollary \ref{cor1.1}, and Theorem \ref{thm:near-fittest-lowepi} are carried out in Section \ref{Sec3} followed by the proof for the multiple peak property in Section \ref{sec:MPP} and the existence of near-fittest paths in Section \ref{sec:NFP}.

\section{Moment Method}\label{sec:mmomentmethod}

In this section, we will employ the second moment approach to computing the limiting free energy in the $N\!K$ model. Our result below shows that it can be explicitly obtained at any temperature if $\alpha\in [\alpha_*,1]$. However, when $0<\alpha<\alpha_*$, it can only be achieved in part of the high temperature regime depending on $\alpha.$

\begin{theorem}\label{thm2}
	The following limits hold.
	\begin{itemize}
		\item[$(i)$] If $\alpha_*\leq \alpha\leq 1,$ then
		\begin{align*}
			\lim_{N\to\infty}\e F_{N,K}(\beta)&= \left\{\begin{array}{ll}
				\ln 2+\frac{\beta^2}{2},&\forall \beta<\beta_c,\\
				\\
				\beta \beta_c,&\forall\beta \geq \beta_c.
			\end{array}\right.
		\end{align*}
		\item[$(ii)$] If $0<\alpha<\alpha_*$, then
		\begin{align*}
			\lim_{N\to\infty}\e F_{N,K}(\beta)&= \ln 2+\frac{\beta^2}{2},\,\,\forall 0<\beta<\sqrt{\ln 2}(1+\sqrt{\alpha}).
		\end{align*}
	\end{itemize}
\end{theorem}

For the rest of this section, we establish this theorem. For any $s\in \mathbb{R},$ define the cardinality of the level set as
\begin{align*}
	L_N(s)&=\bigl|\bigl\{\sigma:H_{N,K}(\sigma)\geq sN\bigr\}\bigr|.
\end{align*}
Then
\begin{align*}
	\e L_N(s)&=2^N\p(H_{N,K}(\mathbbm{1})\geq sN)=2^N\p(z\geq s\sqrt{N})=2^N\Phi(-s\sqrt{N}),
\end{align*}
where $\Phi$ is the cumulative distribution function of the standard normal random variable.
To control the second moment, write \begin{align*}
	&\e L_N(s)^2\\
	&=\sum_{\sigma^1,\sigma^2}\p\bigl(H_{N,K}(\sigma^1)\geq sN,H_{N,K}(\sigma^2)\geq sN\bigr)\\
	&=\sum_{l=0}^N\bigl|\{(\sigma^1,\sigma^2):Q(\sigma^1,\sigma^2)=l/N\}\bigr|\p\bigl(X_1(l/N)\geq s\sqrt{N},X_2(l/N)\geq s\sqrt{N}\bigr),
\end{align*} 
where for any $t\in [0,1]$, $(X_1(t),X_2(t))$ is a bivariate centered normal vector with $\e X_1(t)^2=\e X_2(t)^2=1$ and $\e X_1(t)X_2(t)=t.$ Consider the epistasis overlap of any genome with $\mathbbm{1}=(1,1,\ldots, 1)$, defined as 
\begin{align}
	\label{eqn:def-Qsigma}
	Q(\sigma):=\frac{1}{N}\sum_{i=0}^{N-1}I(\sigma_j=1,\,\,\forall i\leq j\leq i+K).
\end{align}
Observe that for any fixed $\sigma^1,$ the mapping $
\sigma^2\mapsto (\sigma_0^1\sigma_0^2,\ldots,\sigma_{N-1}^1\sigma_{N-1}^2)$
is a bijection so that 
$$\bigl|\{\sigma^2:Q(\sigma^1,\sigma^2)=t\}\bigr|=\bigl|\{\sigma:Q(\sigma)=t\}\bigr|$$ 
and thus,
\begin{align*}
	\bigl|\{(\sigma^1,\sigma^2):Q(\sigma^1,\sigma^2)=t\}\bigr|
	&=\sum_{\sigma^1}\big|\{\sigma^2:Q(\sigma^1,  \sigma^2)=t\}\bigr|\\
	&=\sum_{\sigma}\bigl|\{\sigma:Q(\sigma)=t\}\bigr|=2^N\bigl|\{\sigma:Q(\sigma)=t\}\bigr|.
\end{align*}
Therefore,
\begin{align}\label{momentmethod:eq1}
	\e L_N(s)^2
	&=2^N\sum_{l=0}^N\bigl|\{\sigma:Q(\sigma)=l/N\}\bigr|\p\bigl(X_1(l/N)\geq s\sqrt{N},X_2(l/N)\geq s\sqrt{N}\bigr).
\end{align} 
The next subsection establishes the bounds on the two quantities on the right-hand side of \eqref{momentmethod:eq1}.

\subsection{Some Preliminary Estimates}

\begin{lemma}\label{lem2}
	For any $\sigma\neq \mathbbm{1}$, we have
	\begin{align*}
		Q(\sigma)&\leq \frac{N-K-1}{N}.
	\end{align*}
	Also, for any $1\leq l\leq N-K-1$,
	\begin{align*}
		\bigl|\{\sigma\in \Sigma_N:NQ(\sigma)=l\}\bigr|\leq N2^{N-(K+l)}.
	\end{align*}
\end{lemma}

\begin{proof}
	Without loss of generality, we can assume that $\sigma_{N-1}=-1$ and 
    thus, $I(\sigma_j=1,\,\,\forall k\leq j\leq k+K)=0$ for all  $k=N-K-1, \ldots, N-1$ and thus, $Q(\sigma)\leq (N-K-1)/N.$ 
	
	For the second assertion, we note that the $i$-th summand in the definition of $Q(\sigma)$ is nonzero only if there is a block of at least $(K+1)$ many 1's, starting from $\sigma_i$. Thus for each $\sigma \in \{\pm 1\}^N$, we only need to consider the blocks consisting of at least $(K+1)$ consecutive $+1$'s (modulo $N$) -- see the gray blocks in Figure \ref{fig:fig3} where we represent a genome of length $N$ as a circle due to the periodic boundary condition. Denote the starting indices of these blocks by $i_1, i_2,\ldots, $ and the length of these blocks by $K+n_1, K+n_2, \ldots$, for $n_1, n_2, \cdots \ge 1$. With these notations, we can express.
	\[
	NQ(\sigma) = n_1 + n_2 + \cdots. 
	\]
	Therefore, for any $1\le l\le N-K-1$, we have
	\begin{align*}
		&\bigl|\{\sigma\in \Sigma_N:NQ(\sigma)=l\}\bigr|\\ &\qquad \qquad \le  2^N \sum_{m=1}^l \binom{N}{m} 2^{-m}\sum_{1\le n_1,\ldots, n_m \le l} \mathbbm{1}_{\left\{\sum _{i=1}^m n_i=l\right\}}\prod_{j=1}^m 2^{-K-n_j},
	\end{align*}
	where $m$ is the number of blocks of at least $K+1$ consecutive 1's, the factor $\binom{N}{m}$ bounds the number of choices for the indices $i_1, \ldots, i_m$, and the factor $2^{-m}$ is because the spin right after each block must take the value $-1$ (to terminate the block of consecutive 1's).  
    \begin{figure}[ht]
		\centering
\includegraphics[width=0.4\textwidth]{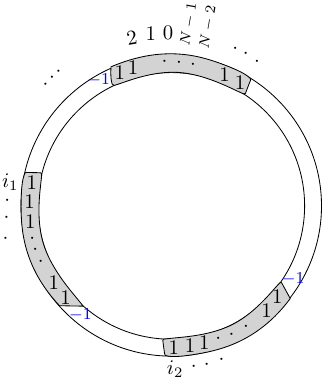}
		\caption{A sketch of the epistasis overlap structure with the genome $\mathbbm 1=(1,1,\ldots,1)$.}
        \label{fig:fig3}
	\end{figure}
	Since $n_1 + \cdots + n_m = l$, the right hand side above can be simplified and bounded by
	\begin{align*}
		2^N \sum_{m=1}^l \binom{N}{m}2^{-m}\sum_{1\le n_1,\ldots, n_m \le l} \mathbbm{1}_{\left\{\sum _{i=1}^m n_i=l\right\}}\cdot 2^{-mK-l}
		\leq 2^{N-l-1} \sum_{m=1}^l a_m
	\end{align*}
	for $$a_m:=\binom{N}{m}\binom{l-1}{m-1}2^{-mK},$$  where we have used the trivial bound $2^{-m}\leq 2^{-1}.$ Note that as long as $N$ is large enough, we have that 
	\begin{align*}
		\frac{a_{m+1}}{a_m} &= \frac{N-m}{m+1}\cdot \frac{l-m}{m} \cdot 2^{-K}\\
		&\le N^2 2^{-K} \le \frac{1}{2}
	\end{align*}
	for all $1\leq m\leq l$.
	Thus, the number of configurations $\sigma$ with $NQ(\sigma)=l$ is bounded by
	\begin{align*}
		\bigl|\{\sigma\in \Sigma_N:NQ(\sigma)=l\}\bigr|&\le 2^{N-l-1}a_1\cdot (1+2^{-1} +2^{-2}+ \cdots ) \\
		&= 2^{N-l} \cdot N 2^{-K}= N2^{N-(K+l)},
	\end{align*}
	which is the desired bound. 
\end{proof}\vspace{4mm}

\begin{remark}\rm
	The exponent in the second upper bound in Lemma \ref{lem2} is indeed tight. To see this, note that $|\{\sigma \in \Sigma_N: NQ(\sigma)=l\}|$ is at least as large as the number of configurations in which there is only one segment of $(K+l)$ consecutive 1's and the rest of the locations do not contain a segment of $K+1$ consecutive 1's. From this, we readily have
	\begin{align}\label{rmk1:eq1}
		\bigl|\{\sigma \in \Sigma_N: NQ(\sigma)=l\}\bigr| \ge N(2^{N-K-l-1-\frac{N-K-l-1}{K+1}}).
	\end{align}
	Here, $N$ counts the choices for the starting position (counting clockwise) of the segment of $(K+l)$ consecutive 1's followed by an additional $-1$ at the end of the segment, separating it from the rest of the $N-K-l-1$ locations. For the remaining locations, setting every integer multiple of $(K+1)$-th location to $-1$ would guarantee that there is no segment of $(K+1)$ consecutive 1's and this requires us to set at most $$
	\bigg\lfloor \frac{N-K-l-1}{K+1}\bigg\rfloor$$ many such locations. These together yield \eqref{rmk1:eq1} and thus,
	\begin{align*}
		\bigl| \{\sigma \in \Sigma_N: NQ(\sigma)=l\}\bigr| &\ge N 2^{N-K-l-1-\frac{N-K}{K+1}}\ge C_\alpha N2^{N-K-l},
	\end{align*}
	where the last inequality holds since $$\frac{N-K}{K+1}\sim \frac{1-\alpha}{\alpha}.$$
\end{remark}

\begin{lemma}\label{lem1}
	For any $t\in [0,1]$ and $s\in \mathbb{R}$, we have that
	\begin{align*}
		\p\bigl(X_1(t)\geq s\sqrt{N},X_2(t)\geq s\sqrt{N}\bigr)&\leq \p(Z\geq s\sqrt{N})^2+\frac{1}{2\pi}\arcsin(t)e^{-sN^2/(1+t)}
	\end{align*}
	for $Z\sim N(0,1).$
\end{lemma}

\begin{proof}
	Our assertion follows as a special case of the general Gaussian comparison inequality established in \cite{li2002normal}.
\end{proof}

\subsection{Moment Comparisons}\label{momentcomparisons}

In light of the first inequality in Lemma \ref{lem2}, we only need to bound the sum in \eqref{momentmethod:eq1} for $l$ in three cases: $\{N\},$ $\{0\}$, $\{1,\ldots,N-K-1\}$, and  It follows that from Lemmas \ref{lem2} and \ref{lem1} and noting that $|\{\sigma:Q(\sigma)=0\}|\leq 2^N,$ we have that
\begin{align*}
	\e L_N(s)^2&\leq 2^N\p(Z\geq s\sqrt{N})+2^{2N}\p(Z\geq s\sqrt{N})^2\\
	&+2^{2N}\p(Z\geq s\sqrt{N})^2+2^N\sum_{i=1}^{N-K-1}N2^{N-(K+l)}\cdot\frac{1}{2\pi}\arcsin(i/N)e^{-Ns^2/(1+i/N)}\\
	&\leq 2^N\p(Z\geq s\sqrt{N})+2^{2N+1}\p(Z\geq s\sqrt{N})^2\\
	&+2^NN\sum_{l=1}^{N-K-1}2^{N-(K+l)}e^{-Ns^2/(1+l/N)},
\end{align*}
where we have used $\arcsin(x)/(2\pi)\leq 1.$ Recall that
$\e L_N(s)=2^N\p(Z\geq s\sqrt{N})$, and  
it follows 
\begin{align*}
	\frac{\e L_N(s)^2}{(\e L_N(s))^2}
	&\leq \frac{2^{-N}}{\p(Z\geq s\sqrt{N})}+2+N\sum_{l=1}^{N-K-1}\frac{2^{-(K+l)}e^{-Ns^2/(1+l/N)}}{\p(Z\geq s\sqrt{N})^2}.
\end{align*}
Note that
\begin{align*}
	\frac{s\sqrt{N}}{\sqrt{2\pi}(s^2N+1)}e^{-s^2N/2}\leq \p (Z\geq s\sqrt{N}).
\end{align*}
From this, the first term can be controlled by
\begin{align*}
	\frac{2^{-N}}{\p(Z\geq s\sqrt{N})}&\leq \frac{\sqrt{2\pi}(s^2N+1)}{s\sqrt{N}}e^{-N(\ln 2-s^2/2)}.
\end{align*}
As for the third term, note that for $t=l/N,$ we can write
\begin{align*}
	N\sum_{l=1}^{N-K-1}\frac{2^{-(K+l)}e^{-Ns^2/(1+l/N)}}{\p(Z\geq s\sqrt{N})^2}&\leq N\cdot\frac{2^3\pi(s^2N+1)^2}{s^2N}\cdot Ne^{- N\min_{0\leq t\leq 1-\alpha}f_s(t)}\\
	&=\frac{2^3\pi(s^2N+1)^2N}{s^2}e^{ -N\min_{0\leq t\leq 1-\alpha}f_s (t)},
\end{align*}
where the first inequality used that $K\geq \alpha N-2$ and
\begin{align*}
	f_s(t)&:=(\alpha+t)\ln 2+\frac{s^2}{1+t}-s^2\\
	&=(\alpha+t)\ln 2-\frac{ts^2}{1+t},\,\,0\leq t\leq 1-\alpha.
\end{align*}
To sum up,\begin{align}\label{momentcomparison:eq1}
	\frac{\e L_N(s)^2}{(\e L_N(s))^2}
	&\leq 2+\frac{\sqrt{2\pi}(s^2N+1)}{s\sqrt{N}}e^{-N(\ln 2-s^2/2)}+\frac{2^2\pi(s^2N+1)^2N}{s^2}e^{ -N\min_{0\leq t\leq 1-\alpha}f_s(t)},
\end{align}
and the right hand side is bounded from above when $s< \sqrt{2\ln 2}=\beta_c$ and $\min_{0\leq t\leq 1-\alpha}f_s(t)>0.$ To verify these inequalities, we divide our discussion into three cases:

\begin{lemma}\label{add:lem1}
	If $2-\sqrt{2}<\alpha\leq 1,$ we have 
	\begin{align*}
		\min_{t\in [0,1-\alpha]}f_s(t)>0,\,\,\forall 0<s<\beta_c.
	\end{align*}
\end{lemma}

\begin{proof} Obviously our assertion holds when $\alpha =1$. Thus, we will only focus on the case $2-\sqrt{2}<\alpha<1$, which implies that $
	\sqrt{\ln 2}<(2-\alpha)\sqrt{\ln 2}<\sqrt{2\ln 2}.$ From this inequality, we divide our discussion into three cases.
	
	\begin{itemize}
		\item  Case I: $\sqrt{\ln 2}< s< (2-\alpha)\sqrt{\ln 2}.$ 
		Note that
		\begin{align*}
			f_s'(t)=\ln 2-\frac{s^2}{(1+t)^2}=0
		\end{align*}
		and it has a unique zero at $t_*=s/\sqrt{\ln 2}-1,$ which lies in $[0,1-\alpha]$. Therefore, $f$ attains its minimum value at $t_*$ with
		$$
		f_s(t_*)=-(1-\alpha) \ln 2+2s\sqrt{\ln 2}-s^2
		$$
		Here, to ensure that the right hand side is strictly positive, it is necessarily that
		\begin{align*}
			\sqrt{\ln 2}(1-\sqrt{\alpha})<s<\sqrt{\ln 2}(1+\sqrt{\alpha}).
		\end{align*}
		Thus, if 
		\begin{align}\label{add:eq1}
			\sqrt{\ln 2}<s<\sqrt{\ln 2}\min(2-\alpha,1+\sqrt{\alpha}),
		\end{align}
		then $\min_{0\leq t\leq 1-\alpha}f_s(t)>0.$ Here, since $1-\alpha<\sqrt{2}-1<(\sqrt{5}-1)/2,$ we see that $2-\alpha<1+\sqrt{\alpha}.$ Hence, whenever $\sqrt{\ln 2}<s<(2-\alpha)\sqrt{\ln 2},$ we have $\min_{t\in [0,1-\alpha]}f_s(t)>0.$
		
		\item Case II:
		$s\leq \sqrt{\ln 2}$. In this regime, we see that $s\leq (1+t)\sqrt{\ln 2}$ for any $0\leq t\leq 1-\alpha$ and hence $f_s'(t)\geq 0$ for all $t\in [0,1-\alpha]$. Hence, $\min_{t\in [0,1-\alpha]}f_s(t)=f_s(0)=\alpha\ln 2>0.$

		\item Case III: $(2-\alpha)\sqrt{\ln 2}\leq s\leq \sqrt{2\ln 2}.$ Since $s\geq (1+t)\sqrt{\ln 2}$ for any $0\leq t\leq 1-\alpha$, we have that $f_s'(t)\leq 0$ for all $t\in [0,1-\alpha]$ and thus, $$\min_{t\in [0,1-\alpha]}f_s(t)=f_s(1-\alpha)=\ln 2-\frac{(1-\alpha) s^2}{2-\alpha},$$ which is strictly positive since $$s
		\leq \sqrt{2\ln 2}<\sqrt{\frac{(2-\alpha)\ln 2}{1-\alpha}}.$$
	\end{itemize}
\end{proof}

\begin{lemma}\label{add:lem2}
	If $\alpha_*\leq \alpha\leq 2-\sqrt{2},$ then
	\begin{align*}
		\min_{t\in [0,1-\alpha]}f_s(t)>0,\,\,\forall 0<s<\beta_c.
	\end{align*}
\end{lemma}

\begin{proof}
	Note that $\alpha\leq 2-\sqrt{2}$ implies that $(2-\alpha)\sqrt{\ln 2}\geq \sqrt{2\ln 2}.$ Hence, as Case I in Lemma \ref{add:lem1}, whenever $$
	\sqrt{\ln 2}<s< \sqrt{\ln 2}\min(2-\alpha,1+\sqrt{\alpha}),$$ 
	we have $\min_{t\in [0,1-\alpha]}f_s(t)>0$. Here, $\alpha\leq 2-\sqrt{2}$ implies that $2-\alpha\geq \sqrt{2}$; $\varepsilon\geq \alpha_*$ implies that $1+\sqrt{\alpha}\geq \sqrt{2}.$ Hence, if $\sqrt{\ln 2}<s<\sqrt{2\ln 2}$, $\min_{t\in [0,1-\alpha]}f_s(t)>0.$ Next, if $s\leq \sqrt{\ln 2},$ we can use Case II in Lemma \ref{add:lem1} to show that $\min_{t\in [0,1-\alpha]}f_s(t)>0$ if $s\leq \sqrt{\ln 2}.$ These complete our proof.
\end{proof}

\begin{lemma}\label{2ndmoment:lem1}
	If $0<\alpha<\alpha_*,$ then
	\begin{align*}
		\min_{t\in [0,1-\alpha]}f_s(t)>0,\,\,\forall 0<s<\sqrt{\ln 2}(1+\sqrt{\alpha})<\beta_c.
	\end{align*}
\end{lemma}

\begin{proof}
	Note that $\alpha\leq \alpha_*$ readily implies that
	\begin{align*}
		2-\alpha\geq 1+\sqrt{\alpha}.
	\end{align*}
	Hence, from Case I in Lemma \ref{add:lem1}, if $\sqrt{\ln 2}<s<\sqrt{2\ln 2}(1+\sqrt{\alpha})$, then $\min_{t\in [0,1-\alpha]}f_s(t)>0.$ Also, from Case II in Lemma \ref{add:lem1}, this strict positivity still holds when $0\leq s\leq \sqrt{\ln 2}.$
\end{proof}

\subsection{Proof of Theorem \ref{thm2}}

Before we begin our proof, we need two lemmas. Lemma \ref{lem3} states that the free energy, the maximal fitness, and the restricted maximal fitness are sufficiently concentrated. 

\begin{lemma}\label{lem3}
	For any $\beta>0$, $N\geq 1$, and $K\geq 1,$ we have that for any $t>0,$
	\begin{align*}
		\p\bigl(|F_{N,K}(\beta)-\e F_{N,K}(\beta)|\geq t\bigr)\leq 2e^{-\frac{Nt^2}{4\beta^2}},\\
		\p\bigl(|M_{N,K}-\e M_{N,K}|\geq t\bigr)\leq 2e^{-\frac{Nt^2}{4}}.
	\end{align*}
	Also, for any nonempty $S\subset \Sigma_N\times\Sigma_N$ and $t>0,$
	\begin{align*}
		\p\bigl(|\overlineit{M}_{N,K}(S)-\e \overlineit{M}_{N,K}(S)|\geq t\bigr)\leq 2e^{-\frac{Nt^2}{16}}.
	\end{align*}
\end{lemma}

\begin{proof}
	Set $$F(x)=\frac{1}{N}\ln \sum_{\sigma}\frac{1}{2^N}\exp \beta\sum_{i=1}^Nx_i(\sigma_i,\ldots,\sigma_{i+K})$$ for any $x=(x_i(\sigma_i,\ldots,\sigma_{i+K}))_{0\leq i\leq N-1,\sigma\in \Sigma_N}$. Then
	\begin{align*}
		F(x)
		&\leq F(y)+\frac{1}{N}\ln \sum_{\sigma}\frac{1}{2^N}e^{\beta \sqrt{N}\|x(\sigma)-y(\sigma)\|_2}\\
		&\leq F(y)+\frac{1}{N}\ln \sum_{\sigma}\frac{1}{2^N}e^{\beta \sqrt{N}\|x-y\|_2}\\
		&=F(y)+\frac{\beta}{\sqrt{N}}\|x-y\|_2,
	\end{align*}
	where $x(\sigma):=(x_i(\sigma_i,\ldots,\sigma_{i+K}))_{0\leq i\leq N-1}$ and we have used that $\|x(\sigma)-y(\sigma)\|_2\leq \|x-y\|_2$ for all $\sigma$. The first assertion follows directly from the Gaussian concentration inequality for Lipschitz functions, see, e.g., \cite[Proposition 1.3.5]{talagrand2010mean}. The second follows from the first by substituting $t$ by $\beta t.$
\end{proof}

\begin{lemma}[Proposition 1.1.3 in \cite{talagrand2010mean}]\label{main:proof:lem2}
	Let $(g_k)_{1\leq k\leq M}$ be a sequence of centered Gaussian r.v. with $\e g_k^2\leq T^2$ for all $1\leq k\leq M.$ Then
	\begin{align*}
		\e\ln \sum_{k=1}^Me^{\beta g_k}\leq \left\{\begin{array}{ll}
			\ln M+\frac{\beta T^2}{2},&\mbox{if $0<\beta<\frac{\sqrt{2\ln M}}{T}$},\\
			\\
			\beta T\sqrt{2\ln M},&\mbox{if $\beta \geq \frac{\sqrt{2\ln M}}{T}$}.
		\end{array}
		\right.
	\end{align*}
	and
	$$
	\e\max_{1\leq k\leq M}g_k\leq T\sqrt{2\ln M}.
	$$
\end{lemma}

\begin{remark}
	\rm In particular, applying this lemma to the free energy of the $\NK$ model, we have the following bounds:
	\begin{align}
		\begin{split}\label{main:proof:thm4}
			\e F_{N,K}(\beta)&\leq 
			\left\{\begin{array}{ll}
				\ln 2+\frac{\beta^2}{2},&\forall \beta<\beta_c,\\
				\\
				\beta \sqrt{2\ln 2},&\forall \beta \geq \beta_c.
			\end{array}
			\right.
		\end{split}
	\end{align}
\end{remark}

Now we are ready to prove Theorem \ref{thm2}.  Assume $\alpha_*\leq \alpha\leq 1.$ From Lemmas \ref{add:lem1} and \ref{add:lem2}, we have that for all $0\leq s<\beta_c,$
\begin{align*}
	\limsup_{N\to\infty}\frac{\e L_N(s)^2}{(\e L_N(s))^2}\leq 2.
\end{align*}
Hence, using the Paley-Zygmund ienquality, as long as $N$ is large enough,
\begin{align*}
	\p(L_N(s)\geq  2^{-1}\e L_N(s))&\geq \frac{(\e L_N(s))^2}{4\e L_N(s)^2 }\geq \frac{1}{9}. 
\end{align*}
Therefore, it follows that with probability at least $1/9,$ $L_N(s)\geq 2^{-1}\e L_N(s)>0$ and thus,
\begin{align*}
	F_{N,K}(\beta)\geq \frac{1}{N}\ln \bigl(2^{-1}\e\bigl( L_N(s) e^{\beta Ns}\bigr)\bigr)=-\frac{\ln 2}{N}+\frac{\ln \e L_N(s)}{N}+\beta s.
\end{align*}
On the other hand, from Lemma \ref{lem3}, for any $\delta>0,$ as long as $N$ is large enough, with probability at least $1-1/10,$ $$
F_{N,K}(\beta)\leq \e F_{N,K}(\beta)+\delta.
$$
Consequently, sending $N\to\infty$ and then $\delta\downarrow 0,$
\begin{align}\label{eq1}
	\liminf_{N\to\infty}\e F_{N,K}(\beta)\geq \lim_{\delta\downarrow 0}\liminf_{N\to\infty}\Bigl(\frac{\ln 2}{N}+\frac{\ln \e L_N(s)}{N}+\beta s-\delta\Bigr)= \ln 2-\frac{s^2}{2}+\beta s.
\end{align}
Next,
note that $s\geq 0\mapsto \ln 2-s^2/2+\beta s$ is a concave function and it maximum value is attained by $s=\beta.$ Thus, if $\beta \leq \beta_c,$ we can take $s=\beta$ so that 
\begin{align*}
	\liminf_{N\to\infty}\e F_{N,K}(\beta)\geq \ln 2+\frac{\beta^2}{2}.
\end{align*}
If $\beta\geq \beta_c$, we send  $s\uparrow \beta_c$ to get
\begin{align*}
	\liminf_{N\to\infty}\e F_{N,K}(\beta)\geq \beta \beta_c.
\end{align*}
These together with \eqref{main:proof:thm4} and Lemma \ref{lem3} conclude the first desired limit $(i)$ in Theorem \ref{thm2}. As for the second case, $\alpha<\alpha_*,$ we also have that
\begin{align*}
	\limsup_{N\to\infty}\frac{\e L_N(s)^2}{(\e L_N(s))^2}\leq 2
\end{align*}
for all $0<s<\sqrt{\ln 2}(1+\sqrt{\alpha})<\beta_c.$ From \eqref{2ndmoment:lem1}, the same argument as above yields that
\begin{align*}
	\liminf_{N\to\infty}\e F_{N,K}(\beta)\geq \ln 2+\frac{\beta^2}{2}.
\end{align*}
as long as $\beta\leq \sqrt{\ln 2}(1+\sqrt{\alpha}).$ This readily implies $(ii)$ in Theorem \ref{thm2}, again due to \eqref{main:proof:thm4} and Lemma \ref{lem3} and our proof is completed.

\begin{remark}\label{remark:momentmethod}\rm
	It is well-known \cite{hashorva2005asymptotics} that for any fixed $0\leq t<1$ and $s>0,$
	$$
	\lim_{N\to\infty}\frac{\p(X_1(t)\geq s\sqrt{N},X_2(t)\geq s\sqrt{N})}{\frac{1}{2\pi s^2N\sqrt{1-t^2}}e^{-\frac{s^2N}{1+t}}}=1.
	$$ 
	This suggests that the leading-order term in \eqref{momentcomparison:eq1} is essentially sharp, revealing an inherent limitation of the current second moment method in obtaining the limiting free energy for all temperatures within the regime $0<\alpha<\alpha_*.$
\end{remark}

\section{Limiting Free Energy and Maximal Fitness}\label{sec6}

This section is devoted to proving Theorem \ref{thm:Fnk-Mnk}, which relies on two key approximations: the $p$-spin $\NK$ model and a balanced multi-species model. The first two subsections provide the necessary groundwork for the main proof.

\subsection{$\NK$ Model with $p$-spin Interaction}

Let $p$ be a fixed even number. For any $N\geq K\geq 1,$ the $\NK$ model with $p$-spin interaction is defined as the Gaussian process $H_{N,K}^p$ with covariance 
\begin{align*}
	\e    H_{N,K}^p(\sigma^1)H_{N,K}^p(\sigma^2)&=\sum_{i=0}^{N-1}\bigl(R_{N,K,i}(\sigma^1,\sigma^2)\bigr)^p
\end{align*}
for $\sigma\in \Sigma_N,$
where $$
R_{N,K,i}(\sigma^1,\sigma^2):=\frac{1}{K+1}\sum_{j=0}^{K}\sigma_{i+j}^1\sigma_{i+j}^2
$$
and as the original $\NK$ model, the spins live on a periodic ring of length $N.$ Set the free energy as
$$
F_{N,K}^p(\beta)=\frac{1}{N}\ln \sum_{\sigma}e^{\beta H_{N,K}^p(\sigma)}.
$$ 
Fix $\alpha\in (0,1]$. From now on, let $K=\lfloor \alpha(N-1)\rfloor$ as in the original $\NK$ model.  If $\alpha=1,$ let $n=k=1.$ Let $1\leq k\leq  n<\infty$ so that 
$
\alpha\leq k/n.
$
Take $l=\lceil N/n\rceil.$ Set $N_l=nl$ and $K_l=kl$. The lemma below establishes a continuity of the free energy in the $p$-spin $\NK$ model in the ratio $K/N$.

\begin{lemma}\label{NKP:lemma1}
	There exists a constant $C>0$ depending only on $\beta$ and $p$ such that
	\begin{align*}
		\limsup_{N\to\infty} \bigl|\e F_{N,K}^p(\beta)-\e F_{N_l,K_l}^p(\beta)\bigr|&\leq \frac{C(k/n-\alpha)}{\alpha}.
	\end{align*}
	
\end{lemma}
\begin{proof} We adapt an interpolation argument to establish our proof.
	Note that $N_l\geq N$ and $K_l\geq K.$ For any $0\leq t\leq 1,$ define
	\begin{align}\label{additional:eq2}		F(t)=\frac{1}{N}\e\ln \sum_{(\sigma,\rho)\in \Sigma_N\times\Sigma_{N_l-N}}\exp\beta\bigl({\sqrt{t}H_{N,K}^p(\sigma)+\sqrt{1-t}H_{N_l,K_l}^p(\sigma,\rho)}\bigr).
	\end{align}
	Denote by $(\sigma^1,\rho^1)$ and $(\sigma^2,\rho^2)$ independent samples from the Gibbs measure associated to this free energy and by $\la \cdot\ra_t$  the corresponding Gibbs expectation.  
	Then
	\begin{align}
		\begin{split}\label{NKP:proof:eq1}
			F(0)&=\frac{1}{N}\e\ln \sum_{\tau\in \Sigma_{N_l}}\exp \beta H_{N_l,K_l}^p(\tau),\\
			F(1)&=\frac{1}{N}\e \ln \sum_{\sigma\in \Sigma_N}\exp{\beta H_{N,K}^p(\sigma)}+\frac{N_l-N}{N}\ln 2.
		\end{split}
	\end{align}
	Now, using the Gaussian integration by parts yields
	\begin{align}
		\nonumber	F'(t)&=\frac{\beta^2}{2}\e\Bigl\la \frac{1}{N}\sum_{i=0}^{N_l-1}R_{N_l,K_l,i}((\sigma^1,\rho^1),(\sigma^2,\rho^2))^p-\frac{1}{N}\sum_{i=0}^{N-1}\bigl(R_{N,K,i}(\sigma^1,\sigma^2)\bigr)^p\Bigr\ra_t\\
		\nonumber	&=\frac{\beta^2}{2}\e\Bigl\la \frac{1}{N}\sum_{i=0}^{N-1}\Bigl(R_{N_l,K_l,i}((\sigma^1,\rho^1),(\sigma^2,\rho^2))^p-\bigl(R_{N,K,i}(\sigma^1,\sigma^2)\bigr)^p\Bigr)\\
		\label{additional:eq1}	&\qquad\qquad\qquad\qquad+\frac{1}{N}\sum_{i=N}^{N_l-1}R_{N_l,K_l,i}((\sigma^1,\rho^1),(\sigma^2,\rho^2))^p\Bigr\ra_t.
	\end{align}
	Here, observe that for all $0\leq i\leq N-1,$ the spins appearing in the overlaps $R_{N_l,K_l,i}((\sigma^1,\rho^1),(\sigma^2,\rho^2))$ and $R_{N,K,i}(\sigma^1,\sigma^2)$ are different from each other by at most $N_l-N+K_l-K$ many locations, it follows that there exists a constant $C$ independent of $i$ such that
	\begin{align*}
		\Bigl| R_{N_l,K_l,i}((\sigma^1,\rho^1),(\sigma^2,\rho^2))-R_{N,K,i}(\sigma^1,\sigma^2)\Bigr|&\leq \frac{C(N_l-N+K_l-K)}{K}.
	\end{align*}
	Hence,
	\begin{align*}
		\bigl|F'(t)\bigr|&\leq C\beta^2p\frac{N_l-N+K_l-K}{K}+\frac{\beta^2}{2}\frac{N_l-N}{N}.
	\end{align*}
	Since this inequality holds for all $t\in (0,1),$ by taking integral and noting \eqref{NKP:proof:eq1}, our assertion follows.
	
\end{proof}

\subsection{Free Energy of a Balanced Multi-Species Model}

Fix an even integer $p\geq 2$ and positive integers $k,n$ with $k<n$.  
Consider the multi-species model defined as follows. Let $\mathscr{S}=\{0,1,\ldots,n-1\}$ be the set of species labels. For any $l\geq 1,$ denote $N=nl.$ Let
$\mathcal{I}_s=\{ls,ls+1,\ldots, ls+l-1\}$ be the set of species belonging to the $s$-th species and set the corresponding overlap as
$$
R_{\MS,s}(\sigma^1,\sigma^2)=\frac{1}{l}\sum_{j=0}^{l-1}\sigma_{ls+j}^1\sigma_{ls+j}^2.
$$
The multi-species model we are interested in has the Hamiltonian  $(H_{\MS,l}^p(\sigma))_{\sigma\in\Sigma_N}$ with covariance
\begin{align}\label{HamiltonianMSp}
	\e H_{\MS,l}^{p}(\sigma^1)H_{\MS,l}^{p}(\sigma^2)=N\xi\bigl(R_{\MS,0}(\sigma^1,\sigma^2),\ldots,R_{\MS,n-1}(\sigma^1,\sigma^2)\bigr),
\end{align}
where for all $x=(x_0,\ldots,x_{n-1})\in \mathbb{R}^n,$
$$
\xi(x):=\frac{1}{n}\sum_{s=0}^{n-1}\Bigl(\frac{x_{s}+x_{s+1}+\cdots+x_{s+k-1}}{k}\Bigr)^p$$
and the indices live on a periodic ring of length $n$ so that $x_{a}=x_{a+n}$ for all $0\leq a\leq n-1.$ When $p=2$, this is an example of the general multi-species Sherrington-Kirkpatrick model in \cite{barra2015multi,panchenko2015free}. Define the free energy as $$
F_{\MS,l}^p(\beta)=\frac{1}{N}\ln\sum_{\sigma\in \Sigma_N}e^{\beta H_{\MS,l}^p(\sigma)}.$$
The following proposition holds.

\begin{proposition}\label{ParisiFormula:NKP}
	We have that for any $\beta>0,$ 
	$
	\lim_{l\to\infty}\e F_{\MS,l}^p(\beta)
	$
	exists.
	Furthermore, 
	\begin{align*}
		\lim_{l\to\infty}\e F_{\MS,l}^p(\beta)&=\ln 2+\frac{\beta^2}{2},\,\,\forall 0<\beta\leq \beta_p,\\
		\lim_{l\to\infty}\e F_{\MS,l}^p(\beta)&\geq \ln2+\beta\beta_p-\frac{\beta_p^2}{2},\,\,\forall \beta>\beta_p,
	\end{align*}
	where $\beta_p>0$ is defined as \begin{align}\label{add:thm10:eq1}
		\beta_p^2=\inf_{0<u<1}(1+u^{-p})I(u)
	\end{align}
	for
	\begin{align*}
		I(u):=\frac{1}{2}\bigl((1+u)\ln (1+u)+(1-u)\ln (1-u)\bigr),\,\,\forall u\in [0,1].
	\end{align*}
\end{proposition}

\begin{proof}
	Note that $\xi$ is a convex and homogeneous polynomial of degree $p$. Furthermore, since $x_a=x_{a+n}$ for all $0\leq a\leq n-1$, the model is balanced in the sense that 
	\begin{align*}
		\frac{\partial\xi}{\partial{x_a}}(1,\ldots,1)
	\end{align*}
	is independent of $0\leq a\leq n-1.$ As a result, from \cite[Example 1.2(c) and Corollary 2.2]{bates2025balanced}, the limiting free energy in the multi-species model is equal to that of a one-dimensional system, more precisely,
	\begin{align*}
		F_{\MS,\infty}^p(\beta):=	\lim_{l\to\infty}\e F_{\MS,l}^p(\beta)=\lim_{N\to\infty}\e F_{N}^p(\beta)=:F_\infty^p(\beta),
	\end{align*}
	where $F_N^p(\beta)$ is the free energy associated to the classical $p$-spin model, whose Hamitonian is given by $H_N^p$ with covariance structure $\e H_N^p(\sigma^1)H_N^p(\sigma^2)=NR(\sigma^1,\sigma^2)^p.$ This establishes the existence of the limit of $\e F_{\MS,l}^p(\beta).$ 
	
	Next,  it is well-known that $F_\infty^p(\beta)$ can be expressed as the Parisi formula and one major consequence of this expression is that it ensures the differentiability of $F_{\infty}^p(\beta)$ for all $\beta>0$, see from \cite{Panchenko2008}. From \cite[Theorem 16.3.1]{talagrand2011mean}, it was also known that $F_\infty^p(\beta)$ is equal to $\ln 2+\beta^2/2$ for all $\beta\leq \beta_p$, where $\beta_p$ is defined \eqref{add:thm10:eq1}. Consequently, 
	\begin{align*}
		F_{\MS,\infty}^p(\beta)=\ln 2+\frac{\beta^2}{2},\,\,\forall 0<\beta<\beta_p.
	\end{align*}
	and using the convexity of $F_{\MS,\infty}^p(\beta),$
	\begin{align*}
		F_{\MS,\infty}^p(\beta)\geq  F_{\MS,\infty}^p(\beta_p)+\frac{dF_{\MS,\infty}^p}{d\beta}(\beta_p)(\beta-\beta_p)=\ln 2+\beta_p\beta-\frac{\beta_p^2}{2},\,\,\forall \beta>\beta_p,
	\end{align*}
	completing our proof.
\end{proof}

\subsection{Proof of Theorem \ref{thm:Fnk-Mnk}}

We now turn to the proof of Theorem \ref{thm:Fnk-Mnk}. First of all, from Remark \ref{main:proof:thm4}, it remains to show the matching lower bound. 
Fix an arbitrary even $p\geq 2.$ Note that the Hamiltonians of our original $\NK$ model and the $p$-spin $\NK$ model satisfy
\begin{align*}
	\e  H_{N,K}(\sigma^1)H_{N,K}(\sigma^2)&=NQ(\sigma^1,\sigma^2)\\
	&\leq N\sum_{i=0}^{N-1}R_{N,K,i}(\sigma^1,\sigma^2)^p=\e H_{N,K}^p(\sigma^1)H_{N,K}^p(\sigma^2)
\end{align*}
and $\e H_{N,K}(\sigma)^2=\e H_{N,K}^p(\sigma)^2=N.$ Using these, it can be argued, similar to the last subsection, by considering the analogous interpolated free energy between $H_{N,K}$ and $H_{N,K}^p$ and computing its derivative similar to \eqref{additional:eq2}  and \eqref{additional:eq1} to deduce that $$
\e F_{N,K}(\beta)\geq \e F_{N,K}^p(\beta).$$ 
It follows from Lemma \ref{NKP:lemma1} that for any two positive integers $k,n$ with $ \alpha\leq k/n\leq  1,$
\begin{align}\label{additional:eq3}
	\liminf_{N\to\infty}\e F_{N,K}(\beta)\geq	\liminf_{N\to\infty}\e F_{N,K}(\beta)\geq \liminf_{l\to\infty}\e F_{N_l,K_l}(\beta)-\frac{C}{\alpha}(k/n-\alpha),
\end{align}
where for $l:=\rceil N/n\rceil$, $N_l=ln$ and $K_l=lk.$

Next, observe that for any $ls\leq i<l(s+1)$ for some $0\leq s\leq n-1,$ we have that $l(s+k)\leq i+K_l<l(s+k+1)$ so that
\begin{align*}
	R_{N_l,K_l,i}(\sigma^1,\sigma^2)
	&=\frac{1}{K_l+1}\sum_{j=0}^{l(s+1)-1-i}\sigma_{i+j}^1\sigma_{i+j}^2+\frac{K_l}{K_l+1}\frac{1}{k}\sum_{r=s+1}^{s+k-1}R_{\MS,r}(\sigma^1,\sigma^2)\\
	&\qquad\qquad+\frac{1}{K_l+1}\sum_{j=0}^{i-ls}\sigma_{l(s+k)+j}^1\sigma_{l(s+k)+j}^2,
\end{align*}
which implies that
\begin{align*}
	\Bigl|R_{N_l,K_l,i}(\sigma^1,\sigma^2)-\frac{1}{k}\sum_{r=s+1}^{s+k-1}R_{\MS,r}(\sigma^1,\sigma^2)\Bigr|&\leq \frac{2l}{K_l}+\frac{1}{K_l}\leq \frac{3}{k}.  
\end{align*}
Denote by 
\begin{align}\label{thm10:add:eq2}
	F(t)=\frac{1}{N}\e\ln \sum_{\sigma}e^{\beta(\sqrt{t}H_{N_l,K_l}^p(\sigma)+\sqrt{1-t}H_{\MS,l}(\sigma))},\,\,t\in [0,1].
\end{align}
Then using the Gaussian integration by parts and $|x^p-y^p|\leq p|x-y|$ for all $x,y\in [-1,1],$ we have
\begin{align*}
	|F'(t)|&=\frac{\beta^2}{2N_l}\Bigl|\e \Bigl\la \sum_{s=0}^{n-1}l\Bigl(\frac{1}{k}\sum_{r=0}^{k-1}R_{\MS,r}(\sigma^1,\sigma^2)\Bigr)^p-\sum_{i=0}^{N_l-1}\bigl(R_{N_l,K_l,i}(\sigma^1,\sigma^2)\bigr)^p\Bigr\ra_{t}\Bigr|\leq \frac{3p\beta^2}{2k},
\end{align*}
where $\la \cdot\ra_t$ is the Gibbs expectation associated to $F(t)$ and $\sigma^1,\sigma^2$ are i.i.d. samples from this measure. Therefore,
\begin{align}\label{NKP:proof:eq2}
	\bigl|\e F_{N_l,K_l}^p(\beta)-\e F_{\MS,l}^p(\beta)\bigr|\leq \frac{3p\beta^2}{2k}.
\end{align}
Consequently, from  \eqref{additional:eq3} and Proposition \ref{ParisiFormula:NKP}, by letting $k,n\to\infty$ along $\alpha\leq k/n\leq 1$ and $k/n\to \alpha,$ we have that for any even $p\geq 2,$
\begin{align*}
	\liminf_{N\to\infty}\e F_{N,K}(\beta)\geq \left\{
	\begin{array}{ll}
		\ln 2+\frac{\beta^2}{2},&\mbox{if $0<\beta\leq \beta_p,$}\\
		\\
		\ln 2+\beta_p\beta-\frac{\beta_p^2}{2},&\mbox{if $\beta>\beta_p.$}
	\end{array}
	\right.
\end{align*}

Finally, our proof will be completed if we can show that $\lim_{p\to\infty}\beta_p=\beta_c.$
Recall $\beta_p$ from \eqref{add:thm10:eq1}. The entropy function $I$ satisfies $I(0)=0,$ $I'(0)=0,$ and $I''(u)=1/(1-u^2)\geq 1.$ Thus,  $
I(u)\geq u^2/2.
$
Together with the fact that $I$ is strictly increasing, we have that for any $0<u_0<1,$
\begin{align*}
	\beta_p^2&\geq \min\Bigl(\inf_{0<u\leq u_0}\frac{(1+u^{-p})u^2}{2},\inf_{u_0\leq u<1}(1+u^{-p})I(u)\Bigr)\\
	&\geq \min\Bigl(\frac{u_0^{-(p-2)}}{2},2I(u_0)\Bigr),
\end{align*}
which implies that
\begin{align*}
	\liminf_{p\to\infty}\beta_p^2\geq 2I(u_0).
\end{align*}
Since this holds for all $0<u_0<1,$ we conclude that $\lim_{p\to\infty}\beta_p^2\geq \lim_{u_0\uparrow 1}2I(u_0)=\beta_c^2$ and thus, $\lim_{p\to \infty}\beta_p^2=\beta_c^2.$ This completes our proof.

\section{Proof of Overlap Gap Properties}\label{Sec3}

In this section, we establish the overlap gap properties as stated in Theorems \ref{thm:Q-high-epi}, \ref{thm:nearfittest-highepi} and \ref{thm:near-fittest-lowepi}. 
For any nonempty $S\subseteq \Sigma_N\times\Sigma_N,$ set
\begin{align*}
	\overlineit{F}_{N,K}(\beta,S)&=    \frac{1}{N}\ln \sum_{(\sigma^1,\sigma^2)\in S}e^{\beta (H_{N,K}(\sigma^1)+H_{N,K}(\sigma^2))}
\end{align*}
and recall $\overlineit{M}_{N,K}(S)$ from \eqref{def:barM}. Note that similar to \eqref{eqn:GSEbounds} and Lemma \ref{lem3}, we also have 
\begin{align}\label{eqn:barGSEbounds}
	\overlineit{M}_{N,K}(S)\leq \frac{1}{\beta}\overlineit{F}_{N,K}(\beta,S)\leq \overlineit{M}_{N,K}(S)+\frac{2\ln 2}{N},\,\,\forall \beta>0
\end{align}
and the following concentration inequality holds, whose proof is omitted the proof.
\begin{lemma}
    \label{concentration:barversion}
     For any $\beta>0$, $N\geq 1$, $K\geq 1,$ and nonempty $S\subset \Sigma_N\times\Sigma_N,$ we have that for any $t>0,$
	\begin{align*}
		\p\bigl(|\overlineit{F}_{N,K}(\beta,S)-\e \overlineit{F}_{N,K}(\beta,S)|\geq t\bigr)\leq 2e^{-\frac{Nt^2}{16\beta^2}},\\
		\p\bigl(|\overlineit{M}_{N,K}(S)-\e \overlineit{M}_{N,K}(S)|\geq t\bigr)\leq 2e^{-\frac{Nt^2}{16}}.
	\end{align*}
\end{lemma}

\subsection{Proof of Theorem \ref{thm:Q-high-epi}}\label{Subsec:thm2}

Assume that $3-2\sqrt{2}<\alpha\leq 1.$ We begin by stating the following useful lemma:

\begin{lemma}[Griffith]
	Let $(F_{N})_{N\geq 1}$ be a sequence of convex and differentiable functions defined on $[0,\infty).$ Assume that $f(x):=\lim_{N\to\infty}F_{N}(x)$ exists for all $x\in [0,\infty).$ If $f$ is differentiable at some $x_0\in (0,\infty),$ then $\lim_{N\to\infty}F_{N}'(x_0)=f'(x_0).$
\end{lemma}

Note that by Gaussian integration by parts, it can be computed directly that $\beta (1-\e\la Q_{1,2}\ra_\beta)=\e F_{N,K}'(\beta).$ From Theorem \ref{thm:Fnk-Mnk} and the Griffith lemma, \eqref{eqn:Q-high-epi-hightemp} holds since
\begin{align}\label{thm1:proof:eq1}
	\lim_{N\to\infty}\beta(1-\e\la Q_{1,2}\ra_\beta) =\lim_{N\to\infty}\e F_{N,K}'(\beta)=F'(\beta)=\left\{
	\begin{array}{ll}
		\beta,&\mbox{if $\beta\leq \beta_c$},\\
		\beta_c,&\mbox{if $\beta>\beta_c$}.
	\end{array}\right.
\end{align}
Next, consider $\beta> \beta_c.$ If $\alpha=1,$ then \eqref{eqn:Q-high-epi-lowtemp-1} and \eqref{eqn:Q-high-epi-lowtemp-2} also hold from \eqref{thm1:proof:eq1} since $Q_{1,2}\in \{0,1\}$. Thus, for the remaining of this part of the proof, we assume that $\alpha_*<\alpha<1.$ 
Note that from Lemma~\ref{lem2}, the values of $Q_{1,2}$ can only be from $\{0,1/N,2/N,\ldots,(N-K)/N,1\}.$ Also, if $t=i/N$ for some $1\leq i\leq N-K$, then  $\e (M_{N,K}(\sigma^1)+M_{N,K}(\sigma^2))^2=2N(1+t)$ for any $(\sigma^1,\sigma^2)$ satisfying $Q_{1,2}=t$ and
\begin{align}\label{thm3.1:proof:eq5}
	\ln |\{Q_{1,2}=t\}|\leq \ln (N2^{2N-K-Nt})=(2N-K-Nt)\ln 2+\ln N.
\end{align}
It follows from Lemma \ref{main:proof:lem2} that
\begin{align}
	\nonumber
	\e F_{N,K}(\beta,\{Q_{1,2}=t\})&\leq \frac{\beta\sqrt{2N(1+t)}\cdot\sqrt{2\ln |\{Q_{1,2}=t\}|}}{N}\\
	\nonumber   &\leq 2\beta\frac{\sqrt{N(1+t)(2N-K-Nt+\ln N)\ln 2}}{N}\\
	\label{thm1:proof:eq5}   &\leq 2\beta \sqrt{(1+t)(2-\alpha-t+\epsilon_N)\ln 2},
\end{align}
where we used $K \geq \alpha(N-1)-1\geq \alpha N-2$ and denoted $\epsilon_N=N^{-1}(2+\ln N).$ Note that
\begin{align}\label{additional:eq10}
	\sup_{0<t\leq 1-\alpha}(1+t)(2-\alpha-t)=\Bigl(\frac{3-\alpha}{2}\Bigr)^2
\end{align}
and the maximum is attained at $t=(1-\alpha)/2.$ From \eqref{thm1:proof:eq5}, we have
\begin{align}
	\nonumber  \max_{t=i/N,1\leq i\leq N-K}\e \overlineit{F}_{N,K}(\beta,\{Q_{1,2}=t\})&\leq \beta (3-\alpha)\sqrt{\ln 2}+2\beta \sqrt{2\epsilon_N\ln 2}\\
	\label{thm1:proof:eq3}    &\leq \beta (3-\alpha)\sqrt{\ln 2}+4\beta \sqrt{\epsilon_N},
\end{align}
where the first inequality used $\sqrt{a+b}\leq \sqrt{a}+\sqrt{b}$ for all $a,b\geq 0,$ while the second one used $2\sqrt{2\ln 2}<4.$
Since the free energies are concentrated as in Lemma \ref{lem3}, by using union bounds, we have that for any $\varepsilon>0$, as long as $N$ is large enough, with probability at least
$
1-2Ne^{-\varepsilon^2N/16\beta^2},
$
\begin{align}\label{additional:eq11}
	\max_{t=i/N,1\leq i\leq N-K}\overlineit{F}_{N,K}(\beta,\{Q_{1,2}=t\})\leq  \beta (3-\alpha)\sqrt{\ln 2}+4\beta \sqrt{\epsilon_N}+\varepsilon.
\end{align}
Next, note that for any $a_1,\ldots,a_k> 0,$
$$
\max(\ln a_1,\ldots,\ln a_k)\leq \ln (a_1+\cdots+a_k)\leq \ln k+\max(\ln a_1,\ldots,\ln a_k)
$$
and that $Q_{1,2}\leq (N-K-1)/N.$ On the same event that \eqref{additional:eq11} holds, we have
\begin{align}
	\nonumber\overlineit{F}_{N,K}(\beta,\{0<Q_{1,2}<1\})&\leq \beta (3-\alpha)\sqrt{\ln 2}+4\beta \sqrt{\epsilon_N}+\varepsilon+\frac{\ln N}{N}\\
	\label{additional:eq8} &=2\beta\beta_c-\beta (2\beta_c-(3-\alpha)\sqrt{\ln 2})+4\beta \sqrt{\epsilon_N}+\varepsilon+\frac{\ln N}{N}.
\end{align}
On the other hand, recall from Theorem \ref{thm:Fnk-Mnk} and Lemma \ref{lem3} that for any $\varepsilon>0$, with probability at least $1-2e^{-N\varepsilon^2/4\beta^2},$ $\beta\beta_c\leq F_{N,K}(\beta)+2\varepsilon.$ From these, we conclude that on an event with probability at least $1-2Ne^{-\varepsilon^2N/16\beta^2}-2e^{-\varepsilon^2N/4\beta^2},$ 
\begin{align}\label{thm1:proof:eq2}
	\la I(0<Q_{1,2}<1)\ra_\beta&\leq Ne^{-\beta N(2\beta_c-(3-\alpha)\sqrt{\ln 2})+N(4\beta \sqrt{\epsilon_N}+5\varepsilon)}.
\end{align}
Since $\alpha>\alpha_*,$ we have $2\beta_c>(3-\alpha)\sqrt{\ln 2}$ and this implies that as long as $\varepsilon$ is small enough,
\begin{align*}
	\lim_{N\to\infty}\e \la I(0<Q_{1,2}<1)\ra_\beta=0.
\end{align*}
Consequently, from \eqref{thm1:proof:eq1},
\begin{align*}
	\lim_{N\to\infty}\e \la I(Q_{1,2}=1) \ra_\beta=\lim_{N\to\infty}\e \la Q_{1,2}\ra_\beta=1-\frac{\beta_c}{\beta},
\end{align*}
Since
\begin{align*}
	\lim_{N\to\infty}\bigl(\e \la I(Q_{1,2}=1) \ra_\beta+\e \la I(Q_{1,2}=0)\ra_\beta\bigr)=1,
\end{align*}
we also have
\begin{align*}
	\lim_{N\to\infty}\e \la I(Q_{1,2}=0) \ra_\beta=\frac{\beta_c}{\beta}.
\end{align*}
These complete the proofs of \eqref{eqn:Q-high-epi-hightemp}, \eqref{eqn:Q-high-epi-lowtemp-1}, and \eqref{eqn:Q-high-epi-lowtemp-2}.

\smallskip

\subsection{Proof of Theorem \ref{thm:nearfittest-highepi}}

Let $\alpha_*<\alpha\leq 1$ be fixed. 
From \eqref{eqn:barGSEbounds} and \eqref{additional:eq8},
with probability at least $1-2Ne^{-\varepsilon^2N/16}$,
\begin{align*}
	\overlineit{M}_{N,K}(\{0<Q_{1,2}<1\})\leq 2\lim_{N\to\infty}\e M_{N,K}-(2\beta_c-(3-\alpha)\sqrt{\ln 2})+4\sqrt{\epsilon_N}+5\varepsilon+\frac{\ln N}{\beta N}
\end{align*}
and this readily leads to \eqref{eqn:nearfittest-highepi-Q}.

Next, we establish \eqref{eqn:nearfittest-highepi-R}. For any admissible value $r\in [-1,1]$ of $R_{1,2}$,
note that by releasing the constraint $Q_{1,2}=0$ and using the symmetry in spin configurations, we have
\begin{align*}
	|\{Q_{1,2}=0,R_{1,2}=r\}|\leq 2^N\Bigl|\Bigl\{\sigma:\frac{1}{N}\sum_{i=0}^{N-1}\sigma_i=r\Bigr\}\Bigr|\leq 2^N\binom{N}{ \frac{N(1+|r|)}{2}}\leq 2^{Nh(|r|)},
\end{align*}
where $h(\cdot)$ is defined at \eqref{eqn:def-h}.
Consequently, from Lemma \ref{main:proof:lem2},
\begin{align}
	\nonumber    \e \overlineit{M}_{N,K}(\{Q_{1,2}=0,R_{1,2}=r\})&\leq 2\sqrt{h(r)\ln 2}\\
	\label{add:thm4:proof:eq1}    &=2\lim_{N\to\infty}\e M_{N,K}-2(\beta_c-\sqrt{h(|r|)\ln 2}).
\end{align}
Hence,
\begin{align}
	\nonumber   \max_{|r|\geq \delta} \e \overlineit{M}_{N,K}(\{Q_{1,2}=0,R_{1,2}=r\})
	&=2\lim_{N\to\infty}\e M_{N,K}-2(\beta_c-\sqrt{h(\delta)\ln 2}),
\end{align}
where the supremum is taken over all admissible $r$ with $|r|\geq \delta$ and we used the fact that $h(\delta)$ is strictly decreasing on $[0,1]$. Now using Lemma \ref{concentration:barversion}, we can argue as before to conclude that
\begin{align}
	\label{additional:eq9}  \limsup_{N\to\infty} \e \overlineit{M}_{N,K}(\{Q_{1,2}=0,|R_{1,2}|\geq \delta\})
	&\leq 2\lim_{N\to\infty}\e M_{N,K}-2(\beta_c-\sqrt{h(\delta)\ln 2}).
\end{align}
Note that for any $0<\delta<1,$
\begin{align*}
	\overlineit{M}_{N,K}(\{\delta<|R_{1,2}|< 1\})&\leq \max\bigl(\overlineit{M}_{N,K}(\{Q_{1,2}=0,|R_{1,2}|>\delta\}),\overlineit{M}_{N,K}(\{0<Q_{1,2}<1\})\bigr).
\end{align*}
From \eqref{eqn:nearfittest-highepi-Q},  \eqref{additional:eq9}, and again Lemma \ref{concentration:barversion} yields \eqref{eqn:nearfittest-highepi-R}.

\subsection{Proof of Corollary \ref{cor1.1} }

Assume $\alpha_*<\alpha\leq 1$. Let $0<\delta<1$ be fixed. Recall $\varepsilon$ from \eqref{cor1.1:eq1}. From Theorem \ref{thm:nearfittest-highepi}, as long as $N$ is large enough, there exists an event $E_{N}$ of probability at least $1-4e^{-\varepsilon^2 N/16}-2e^{-\varepsilon^2 N/4}$ and on $E_N,$
$$
\overlineit{M}_{N,K}(\{0<Q_{1,2}<1\})\leq 2M_{N,K}+3\varepsilon+\varepsilon-7\varepsilon=2M_{N,K}-3\varepsilon
$$
and
$$
\overlineit{M}_{N,K}(\{\delta<|R_{1,2}|<1)\leq 2M_{N,K}+3\varepsilon+\varepsilon-7\varepsilon=2M_{N,K}-3\varepsilon,
$$
where in both inequalities, $3\varepsilon$ comes from Lemmas \ref{lem3} and \ref{concentration:barversion}, and $\varepsilon$ and $7\varepsilon$ arise due to Theorem~\ref{thm:nearfittest-highepi}.
Let $E_N'$ be the event in which there exist some distinct $\sigma,\sigma'\in \mathcal{L}_N(\varepsilon)$ such that either $0<Q(\sigma,\sigma')<1$ or $\delta<|R(\sigma,\sigma')|<1.$ On this event, either
\begin{align*}
	\overlineit{M}_{N,K}(\{0<Q_{1,2}<1\})&\geq 2 M_{N,K}-2\varepsilon
\end{align*}
or
\begin{align*}
	\overlineit{M}_{N,K}(\{\delta<|R_{1,2}|<1\})&\geq 2 M_{N,K}-2\varepsilon.
\end{align*}
This implies that $E_N'\subset E_N^c.$ Consequently, with probability at least $1-4e^{-\varepsilon^2 N/16}-2e^{-\varepsilon^2 N/4}$, $\mathcal{L}_N(\varepsilon)$ satisfies that for any distinct $\sigma,\sigma'\in \mathcal{L}_N(\varepsilon)$, $Q(\sigma,\sigma')=0$ and $|R(\sigma,\sigma')|\leq \delta$.

Next, we continue to bound the cardinality of $\mathcal{L}_N(\varepsilon).$ From Lemma \ref{lem3} and noting that the limiting free energy $F(\beta)$ of the $\NK$ model in Theorem \ref{thm:Fnk-Mnk} is differentiable in $\beta$, it follows from \cite[Theorem 1]{auffinger2018concentration} that for every $\beta>0$ and $0<\eta<\beta_c,$ there exist positive constants $\Gamma(\beta,\eta),$ $C(\beta,\eta),$ and $N(\beta,\eta)$ such that for all $N\geq N(\beta,\eta)$, with probability at least $1-C(\beta,\eta)e^{-N/C(\beta,\eta)},$ we have
\begin{align*}
	F(\beta)-\beta\eta\leq  F_{N,K}(\beta)\leq F(\beta)+\beta\eta
\end{align*}
and
\begin{align*}
	\Bigl\la I\Bigl(\Bigl|\frac{H_{N,K}(\sigma)}{N}-F'(\beta)\Bigr|\geq \eta\Bigr) \Bigr\ra\leq e^{-\Gamma(\beta,\eta)N},
\end{align*}
which after taking complement and logarithm leads to
\begin{align*}
	\frac{1}{N}\ln (1-e^{-\Gamma(\beta,\eta)N})+F(\beta)-\beta\eta&\leq \frac{1}{N}\ln (1-e^{-\Gamma(\beta,\eta)N})+F_{N,K}(\beta)\\
	&\leq \frac{1}{N}\ln\sum_{|H_{N,K}(\sigma)/N-F'(\beta)|< \eta} e^{\beta H_{N,K}(\sigma)}\\
	&\leq F_{N,K}(\beta)\leq F(\beta)+\beta\eta
\end{align*}
and thus,
\begin{align}
	\begin{split}\label{additional:eq4}
		\Bigl|\frac{1}{N}\ln\Bigl|\Bigl\{\sigma:\Bigl|\frac{H_{N,K}(\sigma)}{N}-F'(\beta)\Bigr|\leq \eta\Bigr\}\Bigr|-(F(\beta)-\beta F'(\beta))\Bigr|\leq 2\beta\eta-\frac{1}{N}\ln(1-e^{-\Gamma(\beta,\eta)N}).
	\end{split}
\end{align}
Note that by Theorem \ref{thm:Fnk-Mnk}, $F'(\beta)=\beta$ for all $0<\beta\leq \beta_c.$ 
From these, for any $0<\beta_1<\beta_2<\beta_c,$ letting $\eta=(\beta_2-\beta_1)/2$ and $\beta=(\beta_1+\beta_2)/2,$ \eqref{additional:eq4} can be written as
\begin{align}
	\begin{split}\label{additional:eq5}
		&\Bigl|\frac{1}{N}\ln\Bigl|\Bigl\{\sigma:\beta_1\leq \frac{H_{N,K}(\sigma)}{N}\leq \beta_2\Bigr\}\Bigr|-\Bigl(\ln 2-\frac{1}{2}\Bigl(\frac{\beta_1+\beta_2}{2}\Bigr)^2\Bigr)\Bigr|\\
		&\leq \frac{1}{2}(\beta_2^2-\beta_1^2)-\frac{1}{N}\ln(1-e^{-\Gamma(\beta,\eta)N}).
	\end{split}
\end{align}
For any $\eta',\eta''>0,$ from \eqref{additional:eq4} with $\beta=\beta_c$ and $\eta=\eta'$ and from \eqref{additional:eq5}, it follows by an covering argument that there exist positive constants $C'=C'(\varepsilon,\eta',\eta
'')$ and $N'=N'(\varepsilon,\eta',\eta'')$ such that for every $N\geq N'$, with probability at least $1-C'e^{-N/C'},$
\begin{align*}
	&\Bigl|\frac{1}{N}\ln\Bigl|\Bigl\{\sigma:\beta_c-\varepsilon\leq \frac{H_{N,K}(\sigma)}{N}\leq \beta_c+\eta'\Bigr\}\Bigr|-\max_{\beta_c-\varepsilon\leq \beta\leq \beta_c}\Bigl(\ln 2-\frac{\beta^2}{2}\Bigr)\Bigr|\leq 2\beta_c\eta'+\eta'',
\end{align*}
where $$
\max_{\beta_c-\varepsilon\leq \beta\leq \beta_c}\Bigl(\ln 2-\frac{\beta^2}{2}\Bigr)=\ln 2-\frac{1}{2}(\beta_c-\varepsilon)^2=\varepsilon\Bigl(\beta_c-\frac{\varepsilon}{2}\Bigr).
$$
Finally, note that from Lemma \ref{lem3} and Theorem \ref{thm:Fnk-Mnk}, if $\beta_c$ is substituted by $M_{N,K}$, the same inequality remains true with a different constant $C'$. This completes our  proof.

\subsection{Proof of Theorem \ref{thm:near-fittest-lowepi}}

Assume $0<\alpha\leq \alpha_*.$ Note that for $\beta\leq \beta_c$, \eqref{eqn:Q-high-epi-hightemp} holds by the same argument in Subsection \ref{Subsec:thm2}. For $\beta> \beta_c,$ the inequality \eqref{thm1:proof:eq5} still holds, but instead of \eqref{additional:eq10}, we now use that for any $0<\delta<c_1(\alpha)$, letting $\eta=\delta+2^{-1}\Delta(\alpha ),$ we have
\begin{align*}
	\max_{t\in [0,1-\alpha]:\bigl|t-\frac{1-\alpha}{2}\bigr|\geq \eta}(1+t)(2-\alpha-t)=\Bigl(\frac{3-\alpha}{2}\Bigr)^2-\eta^2 
\end{align*}
so that as in the proof of Theorem \ref{thm:Q-high-epi},
\begin{align}
	\nonumber    &\limsup_{N\to\infty}\e \overlineit{F}_{N,K}\Bigl(\beta,\Bigl(0,\frac{1-\alpha}{2}-\delta\Bigr]\bigcup\Bigl[\frac{1-\alpha}{2}+\delta,1\Bigr)\Bigr)\\
	\nonumber     &\leq 2\lim_{N\to\infty}\e F_{N,K}(\beta)-2\beta\Bigl(\beta_c-\sqrt{\Bigl(\Bigl(\frac{3-\alpha}{2}\Bigr)^2-\eta^2\Bigr)\ln 2}\Bigr)\\
	\label{additional:eq12}    &=2\lim_{N\to\infty}\e F_{N,K}(\beta)-2\beta\beta_c\Bigl(1-\sqrt{1-\frac{\delta(\Delta(\alpha)+\delta)}{2}}\Bigr),
\end{align}
where the last equality holds since
\begin{align*}
	\Bigl(\frac{3-\alpha}{2}\Bigr)^2-\eta^2&=\Bigl(\Bigl(\frac{3-\alpha}{2}\Bigr)^2-2\Bigr)+2-\eta^2\\
	&=\frac{\Delta(\alpha)^2}{4}+2-\Bigl(\frac{\Delta(\alpha)^2}{4}+\delta \Delta(\alpha)+\delta^2\Bigr)=2-\delta\bigl(\Delta(\alpha)+\delta\bigr).
\end{align*}
Observe that for any $0<\delta<c_1(\alpha),$
\begin{align*}
	0<\delta(\Delta(\alpha)+\delta)<c_1(\alpha)c_2(\alpha)=\alpha.
\end{align*}
Hence, we have
\begin{align*}
	\limsup_{N\to\infty}\e\Bigl\la I\Bigl(Q_{1,2}\in\Bigl(0,\frac{1-\alpha}{2}-\eta\Bigr]\bigcup\Bigl[\frac{1-\alpha}{2}+\eta,1\Bigr)\Bigr)\Bigr\ra_\beta=0.
\end{align*}
This establishes \eqref{thme.1:eq0} by noting that
$$
\frac{1-\alpha}{2}-\eta=c_1(\alpha)-\delta\,\,\mbox{and}\,\,\frac{1-\alpha}{2}+\eta=c_2(\alpha)+\delta.
$$
Also, from \eqref{eqn:GSEbounds}, \eqref{eqn:barGSEbounds}, and \eqref{additional:eq12}, we have  \eqref{thm:low-epi:eq1}.

Lastly for \eqref{thm3:eq3}, let $\delta\in (0,1)$ satisfy $
h(\delta)<1/(1-\alpha/2).
$
Note that for any $\delta'\in (0,c_1(\alpha))$ and admissible $r,$ we have
$$
|\{0\leq Q_{1,2}\leq c_2(\alpha)+\delta',R_{1,2}=r\}|\leq 2^{Nh(|r|)}.
$$
We can argue similar to \eqref{thm1:proof:eq5} by using Lemma \ref{main:proof:lem2} to get that
\begin{align*}
	&\limsup_{N\to\infty}\e \overlineit{M}_{N,K}(\{0\leq Q_{1,2}\leq c_2(\alpha)+\delta',|R_{1,2}|\geq \delta\})\\
	&\leq 2\sqrt{(1+c_2(\alpha)+\delta')h(\delta)\ln 2}\\
	&=2\lim_{N\to\infty}\e M_{N,K}-2\beta_c\Bigl(1-\sqrt{\frac{1+c_2(\alpha)+\delta'}{2}h(\delta)}\Bigr).
\end{align*}
Since
\begin{align*}
	\overlineit{M}_{N,K}(\{\delta<|R_{1,2}|<1\})&\leq \max\bigl(\overlineit{M}_{N,K}(\{c_2(\alpha)+\delta'<Q_{1,2}<1\}),\\
	&\qquad\quad
	\overlineit{M}_{N,K}(\{0\leq Q_{1,2}\leq c_2(\alpha)+\delta',\delta<|R_{1,2}|<1\})\bigr),
\end{align*}
using Lemma \ref{concentration:barversion} and \eqref{thm:low-epi:eq1}, we conclude that
\begin{align*}
	&\limsup_{N\to\infty}\e \overlineit{M}_{N,K}(\{\delta<|R_{1,2}|<1\})\\
	&\leq 2\lim_{N\to\infty}\e M_{N,K}\\
	&-2\beta_c\min\Bigl(1-\sqrt{1-\frac{\delta'(\Delta(\alpha)+\delta')}{2}},1-\sqrt{\frac{(1+c_2(\alpha)+\delta')}{2}h(\delta)}\Bigr).
\end{align*}
Sending $\delta'\uparrow c_1(\alpha),$ we see that
\begin{align*}
	\lim_{\delta'\uparrow c_1(\alpha)}\delta'(\Delta(\alpha)+\delta')&=c_1(\alpha)c_2(\alpha)
	=\frac{1}{4}((1-\alpha)^2-(\alpha^2-6\alpha+1))=\alpha
\end{align*}
and
\begin{align*}
	\lim_{\delta'\uparrow c_1(\alpha)}   ( 1+c_2(\alpha)+\delta')=1+c_1(\alpha)+c_2(\alpha)=2-\alpha.
\end{align*}
Hence,
\begin{align*}
	&\limsup_{N\to\infty}\e \overlineit{M}_{N,K}(\{\delta<|R_{1,2}|<1\})\\
	&\leq 2\lim_{N\to\infty}\e M_{N,K}
	-2\beta_c\min\Bigl(1-\sqrt{1-\frac{\alpha}{2}},1-\sqrt{\Bigl(1-\frac{\alpha}{2}\Bigr)h(\delta)}\Bigr)\\
	&=2\lim_{N\to\infty}\e M_{N,K}
	-2\beta_c\Bigl(1-\sqrt{\Bigl(1-\frac{\alpha}{2}\Bigr)h(\delta)}\Bigr),
\end{align*}
where the last equality holds since $h(\cdot)\geq 1$. Hence, 
\eqref{thm3:eq3} holds.

\section{Multiple-Peak Property}\label{sec:MPP}
In this section, we prove the multiple-peak phenomenon of the $\NK$ model.
Let $(H_{N,K}^1(\sigma))_{\sigma\in \Sigma_N}$ and $(H_{N,K}^2(\sigma))_{\sigma\in \Sigma_N}$ be i.i.d. copies of $(H_{N,K}(\sigma))_{\sigma\in \Sigma_N}$. For any $0\leq s\leq 1,$ set
\begin{align*}
	H_{N,K,s}^1(\sigma)&=\sqrt{s}H_{N,K}(\sigma)+\sqrt{1-s}H_{N,K}^1(\sigma),\\
	H_{N,K,s}^2(\sigma)&=\sqrt{s}H_{N,K}(\sigma)+\sqrt{1-s}H_{N,K}^2(\sigma).
\end{align*}
Denote by $Z_{N,K,s}^1(\beta)$ and $Z_{N,K,s}^2(\beta)$ the partition functions and by $\sigma^{s,1}$ and $\sigma^{s,2}$ the fittest genomes corresponding to $H_{N,K,s}^1$ and $H_{N,K,s}^2$, respectively. Define
\begin{align*}
	\phi(s)&=\e Q(\sigma^{s,1},\sigma^{s,2}).
\end{align*}

\begin{lemma}\label{holder}
	For any $0<s_1<s_0\leq 1,$ we have that
	$$\phi(s_0)\leq \phi(1)^{1-\frac{\ln s_0}{\ln s_1}}\phi(s_1)^{\frac{\ln s_0}{\ln s_1}}.$$
\end{lemma}

\begin{proof}
	Denote by $\la \cdot\ra_{\beta,s}$ the Gibbs expectation with respect to $(\sigma^1,\sigma^2)$ sampled from the Gibbs measure associated to the partition function, $Z_{N,K,s}^1(\beta)\times Z_{N,K,s}^2(\beta).$ Define $
	\phi_\beta(s)=\e \la Q(\sigma^1,\sigma^2)\ra_{\beta,s}$
	for $s\geq 0.$  By applying \cite[Lemma 10.3]{chatterjee2014superconcentration} to $v\geq 0\mapsto \phi_\beta(e^{-v})$, we have that for any $0\leq v_0\leq v_1<\infty,$
	\begin{align*}
		\phi_\beta(e^{-v_0})&\leq \phi_\beta(1)^{1-\frac{v_0}{v_1}}\phi_\beta(e^{-v_0})^{\frac{v_0}{v_1}}.
	\end{align*}
	From this inequality, taking $v_0=-\ln s_0 $ and $v_1=-\ln {s_1}$ and then sending $\beta\to\infty$ complete our proof. 
\end{proof}

Now we turn to the proof of Theorem \ref{thm:peaks}. Fix an arbitrary $0<\varepsilon<1$. 
Fix $0<s_0<1$ such that
\begin{align}\label{add:thm4:proof:eq6}
	\Bigl(\frac{1}{\sqrt{s_0}}-1+\sqrt{\frac{1-s_0}{s_0}}\Bigr)(\sqrt{2\ln 2}+1)<\frac{\varepsilon}{4}.
\end{align}
Note that from Lemmas \ref{lem3} and \ref{main:proof:lem2}, there exists a positive constant $C$ such that
\begin{align}\label{thm4:proof:eq2}
	\p\bigl(M_{N,K}\leq (\sqrt{2\ln 2}+1),|M_{N,K}-\e M_{N,K}|\leq 4^{-1}\varepsilon\bigr)\geq 1-Ce^{-N/C}.
\end{align}
Next, for any $S\subset [0,1]$ and $s\in [0,1]$, define
\begin{align*}
	\overlineit{M}_{N,K,s}(S)&=\max_{(\sigma^1,\sigma^2)\in S}\Bigl(\frac{H_{N,K}^1(\sigma^1)}{N}+\frac{H_{N,K}^2(\sigma^2)}{N}\Bigr).
\end{align*}
Note that from \eqref{thm3.1:proof:eq5}
and
$\e (H_{N,K,s}^1(\sigma^1)+H_{N,K,s}^2(\sigma^2))^2=2N(1+ts),
$ 
Lemma \ref{main:proof:lem2} implies that for any $t=i/N$ with $1\leq i\leq N-K,$
\begin{align}\label{thm4:proof:eq4}
	\e \overlineit{M}_{N,K,s}(\{Q_{1,2}=t\})\leq 2\sqrt{(1+ts)(2-\alpha-t+\epsilon_N)\ln 2},
\end{align}
where $\epsilon_N=(2+\ln N)/N.$
To control the right-hand side, notice that there exist some $0<s_1<s_0$ small enough so that as long as $N$ is large enough, for any $t\in (0,1-\alpha]$,
\begin{align*}
	2\sqrt{(1+ts_1)(2-\alpha-t+\varepsilon_N)\ln 2}\leq 2\sqrt{(1+s_1)(2-\alpha/2)\ln 2}<2\sqrt{2\ln 2}.
\end{align*}
Consequently, similar to the proofs of \eqref{eqn:Q-high-epi-lowtemp-1} and \eqref{eqn:Q-high-epi-lowtemp-2}, 
\begin{align}\label{thm4:proof:eq5}
	\limsup_{N\to\infty}\e\overlineit{M}_{N,K,s_1}(\{0<Q_{1,2}<1\})<2\lim_{N\to\infty}\e M_{N,K}.
\end{align}
Also, from Lemma \ref{main:proof:lem2},
\begin{align*}
	\limsup_{N\to\infty} \e \overlineit{M}_{N,K,s_1}(\{Q_{1,2}=1\})\leq 2\sqrt{(1+s_1)\ln 2}< 2\lim_{N\to\infty}\e M_{N,K}.
\end{align*}
Putting these two together yields
\begin{align}\label{thm4:proof:eq1}
	\limsup_{N\to\infty}\e \overlineit{M}_{N,K,s_1}(\{Q_{1,2}>0\})<2\lim_{N\to\infty}\e M_{N,K}.
\end{align}
Since $\overlineit{M}_{N,K,s_1}(\{Q_{1,2}>0\})$ and $M_{N,K}$ are concentrated and $M_{N,K},M_{N,K,s_1}^1,M_{N,K,s_1}^2$ are identically distributed, there exist some $\eta,C'>0$ such that with probability at least $1-C'e^{-N/C'},$
\begin{align*}
	\overlineit{M}_{N,K,s_1}(\{Q_{1,2}>0\})\leq M_{N,K,s_1}^1+M_{N,K,s_1}^2-\eta.
\end{align*}
Therefore, with probability at least $1-C'e^{-N/C'},$ $Q(\sigma^{s_1,1},\sigma^{s_1,2})=0,$ where $\sigma^{s_1,1}$ and $\sigma^{s_1,2}$ are the fittest genomes of $H_{N,K,s_1}^1$ and $H_{N,K,s_1}^2$ respectively. 
Consequently, from Lemma \ref{holder},
\begin{align*}
	\e Q(\sigma^{s_0,1},\sigma^{s_0,2})&\leq (C'e^{-N/C'})^{\frac{\ln s_0}{\ln s_1}}.
\end{align*}
It follows by the Markov inequality,
\begin{align}\label{add:thm4:proof:eq2}
	\p(Q(\sigma^{s_0,1},\sigma^{s_0,2})>0)&=\p(Q(\sigma^{s_0,1},\sigma^{s_0,2})\geq N^{-1})\leq N(C'e^{-N/C'})^{\frac{\ln s_1}{\ln s_0}}.
\end{align}
Now, similar to \eqref{add:thm4:proof:eq1}, 
\begin{align*}
	\max_{|r|\geq \varepsilon}\e \overlineit{M}_{N,K,s_0}(\{Q_{1,2}=0,R_{1,2}=r\})&\leq 2\sqrt{h(\varepsilon)\ln 2}<2\sqrt{2\ln 2}.
\end{align*}
Hence, since $\overlineit{M}_{N,K,s_0}(\{Q_{1,2}=0,R_{1,2}=r\})$ is concentrated, there exist positive constants $\eta'$ and $C''$ so that with probability at least $1-C''e^{-N/C''},$
\begin{align*}
	\overlineit{M}_{N,K,s_0}(\{Q_{1,2}=0,|R_{1,2}|\geq \varepsilon\})<M_{N,K,s_0}^1+M_{N,K,s_0}^2-\eta'.
\end{align*}
In particular, this and \eqref{add:thm4:proof:eq2} readily imply that
\begin{align}\label{thm4:proof:eq7}
	\p(Q(\sigma^{s_0,1},\sigma^{s_0,2})=0,|R(\sigma^{s_0,1},\sigma^{s_0,2})|< \varepsilon)\geq 1-\Delta_N,
\end{align}
where $$
\Delta_N:=N(C'e^{-N/C'})^{\frac{\ln s_1}{\ln s_0}}+C''e^{-N/C''}.$$

Finally, denote by $(\sigma^{s_0,\ell})_{\ell\geq 1}$ a sequence of fittest genomes associated to $H_{N,K,s_0}^\ell$, defined in a similar manner as $H_{N,K,s_0}^1$ and $H_{N,K,s_0}^2.$ Let $S_N(\varepsilon)$ be the collection of all $\sigma^{s_0,\ell}$ for
$$1\leq \ell\leq \min\bigl\{\Delta_N^{-1/4},(4Ce^{N/C})^{1/2}\bigr\}=:L_N.$$
From  \eqref{thm4:proof:eq2} and \eqref{thm4:proof:eq7}, with probability at least $$1-\Delta_N|S_N(\varepsilon)|^2-4Ce^{-N/C}|S_N(\varepsilon)|\geq 1-\sqrt{\Delta_N}-\sqrt{4Ce^{N/C}},$$ we have that for any two distinct $\ell,\ell'\in S_N(\varepsilon),$
\begin{align*}
	Q(\sigma^{s_0,\ell},\sigma^{s_0,\ell'})=0\,\,\mbox{and}\,\,|R(\sigma^{s_0,\ell},\sigma^{s_0,\ell'})|< \varepsilon.
\end{align*}
and that  for any $\ell\in S_N(\varepsilon),$
\begin{align*}
	\frac{H_{N,K}({\sigma^{s_0,\ell}})}{N}&= \frac{H_{N,K,s_0}^\ell(\sigma^{s_0,\ell})}{N}+\Bigl(\frac{1}{\sqrt{s_0}}-1\Bigr)\frac{H_{N,K,s_0}^\ell(\sigma^{s_0,\ell})}{N}-\frac{\sqrt{1-s_0}}{\sqrt{s_0}}\frac{H_{N,K}^\ell(\sigma^{s_0,\ell})}{N}\\
	&\geq M_{N,K,s_0}^\ell-\Bigl(\frac{1}{\sqrt{s_0}}-1+\frac{\sqrt{1-s_0}}{\sqrt{s_0}}\Bigr)(\sqrt{2\ln 2}+1)\\
	&\geq M_{N,K}-\varepsilon,
\end{align*}
where \eqref{add:thm4:proof:eq6} was used in the last inequality. Noting that $L_N$ is of exponential order completes our proof.

\section{Existence of Near-Fittest Paths}\label{sec:NFP}
This final section is dedicated to establishing the existence of the near-fittest paths, i.e., Theorem~\ref{thm:near-fittest-path}.
First of all, note that although we focus on the $\NK$ model where $N$ and $K$ satisfy the relation $K=\lfloor \alpha(N-1)\rfloor$, it can still be defined in the same way as before for any given integers $N,K\geq 1$ without this relation. In particular, whenever $K\geq N-1$, this model is the same as the REM. The following lemma establishes the monotonicity of $\e M_{N,K}$ in $K$. 

\begin{lemma}\label{monotonicity}
	For any $N\geq 1$ and $1\leq K_1\leq K_2$ we have that $
    \e M_{N,K_1}\leq \e M_{N,K_2}.$
   
\end{lemma}

\begin{proof}
	Observe that obviously, for any $\sigma^1, \sigma^2 \in \Sigma_N$,
	\begin{align*}
		\e H_{N,K_1}(\sigma^1)H_{N,K_1}(\sigma^2)\geq \e H_{N,K_2}(\sigma^1)H_{N,K_2}(\sigma^2)
	\end{align*}
	and the two sides equal each other when $\sigma^1=\sigma^2$. By Slepian's lemma, our assertion holds.
	
\end{proof}

Recall the disorders $X_i(\sigma_i,\ldots,\sigma_{i+K})$ in $H_{N,K}.$ For any $N_1,N_2\geq 1$ with $N_1+N_2=N,$ write
\begin{align*}
	\tau&=(\tau_0,\ldots,\tau_{N_1-1})\,\,\mbox{and}\,\,\rho=(\rho_0,\ldots,\rho_{N_2-1}).
\end{align*}
Set 
\begin{align}
	\begin{split}\label{bridge:eq11}
		V_{N_1,K}^1(\tau)&=\sum_{i=0}^{N_1-K-1}X_i(\tau_i,\ldots,\tau_{i+K})\,\,\mbox{if $N_1\geq K+1$ and $V_{N_1,K}^1(\tau)=0$ if $N_1< K+1$},\\
		V_{N_2,K}^2(\rho)&=\sum_{i=0}^{N_2-K-1}X_{N_1+i}(\rho_i,\ldots,\rho_{i+K})\,\,\mbox{if $N_2\geq K+1$ and $V_{N_2,K}^2(\rho)=0$ if $N_2< K+1$}.
	\end{split}
\end{align}
Note that these functions are not the Hamiltonians for the $\NK$ models under our setting since their spins do not have the cyclic structures. From now on, we let $K=\lfloor \alpha(N-1)\rfloor.$

\begin{lemma}\label{bridge:lem1}
	Let $\eta,\alpha\in (0,1)$ satisfy \eqref{bridge:eq14}.
	There exists some $N_0\in \mathbb{N}$ depending on $\alpha$ and $\eta$ such that for any  $c\in (0,1)$ and $N,N_1,N_2\in \mathbb{N}$ satisfying 
	\begin{align}\label{bridge:condition}
		N_1+N_2=N\geq N_0\,\,\mbox{and}\,\,        \min\Bigl(\frac{N_1}{N},\frac{N_2}{N}\Bigr)\geq c,
	\end{align}
	we have that with probability at least $1-\omega e^{-\eta^2c^2N/\omega},$ whenever $\sigma_*\in \Sigma_N$ satisfies
	\begin{align}\label{bridge:ass1}
		\frac{H_{N,K}(\sigma^*)}{N}\geq M_{N,K}-\eta,
	\end{align}
	the following inequality holds
	\begin{align}\label{bridge:lem1:eq1}
		\min\Bigl(\frac{V_{N_1,K}^1(\tau^*)}{N_1},\frac{V_{N_2,K}^2(\rho^*)}{N_2}\Bigr)\geq M_{N,K}-\frac{4\eta}{c},
	\end{align}
	where $\tau^*:=(\sigma_0^*,\ldots,\sigma_{N_1-1}^*)$ and $\rho^*:=(\sigma_{N_1}^*,\ldots,\sigma_{N-1}^*).$ Here, $\omega>0$ is an absolute constant independent of all other variables.
\end{lemma}

\begin{proof}
We divide our proof into two cases.

{\noindent \bf Case I: $N_1\geq K+1$ and $N_2\geq K+1$.}
    For any $\sigma\in \Sigma_N,$ write
	\begin{align*}
		H_{N,K}(\sigma)&=V_{N_1,K}^1(\tau)+E_N^1(\sigma_{N_1-K},\ldots,\sigma_{N_1+K-1})\\
		&+V_{N_2,K}^2(\rho)+E_{N}^2(\sigma_{N-K},\ldots,\sigma_{N+K-1}),
	\end{align*}
	where we set
	\begin{align}
		\begin{split}\label{bridge:eq2}
			\tau&=(\tau_0,\ldots,\tau_{N_1-1})=(\sigma_0,\ldots,\sigma_{N_1-1}),\\    
			\rho&=(\rho_0,\ldots,\rho_{N_2-1})=(\sigma_{N_1},\ldots,\sigma_{N-1}),
		\end{split}
	\end{align}
	and
	\begin{align*}
		E_N^1(\sigma_{N_1-K},\ldots,\sigma_{N_1+K-1})&=\sum_{i=N_1-K}^{N_1-1}X_{i}(\sigma_i,\ldots,\sigma_{i+K}),\\
		E_{N}^2(\sigma_{N-K},\ldots,\sigma_{N+K-1})&=\sum_{i=N-K}^{N-1}X_{i}(\sigma_i,\ldots,\sigma_{i+K}).
	\end{align*}
	Let ${H}_{N_1,K}^1$ and ${H}_{N_2,K}^2$ be the Hamiltonians for the $\NK$ models associated to the parameter pairs $(N_1,K)$ and $(N_2,K)$, respectively.  Then for any $\tau\in\Sigma_{N_1}$ and $\rho\in \Sigma_{N_2}$, we can write them as
	\begin{align*}
		{H}_{N_1,K}^1(\tau)&=V_{N_1,K}^1(\tau)+\tilde{E}_N^1(\tau_{N_1-K},\ldots,\tau_{N_1+K-1}),\\
		{H}_{N_2,K}^2(\rho)&=V_{N_2,K}^2(\rho)+\tilde{E}_N^2(\rho_{N_2-K},\ldots,\rho_{N_2+K-1}),
	\end{align*}
	where
	\begin{align*}
		\tilde{E}_N^1(\tau_{N_1-K},\ldots,\tau_{N_1+K-1})&=\sum_{i=N_1-K}^{N_\ell-1}\tilde{X}_i^1(\tau_i,\ldots,\tau_{i+K}), \\
		\tilde{E}_N^2(\rho_{N_2-K},\ldots,\rho_{N_2+K-1})&=\sum_{i=N_2-K}^{N_2-1}\tilde{X}_i^2(\rho_i,\ldots,\rho_{i+K})
	\end{align*}
	for i.i.d. standard normal random variables $\tilde{X}_i^\ell$, $\ell=1,2$, which are independent of all other randomness. Let
	\begin{align*}
		\Lambda_N&=\max_{\sigma}\Bigl|\frac{H_{N,K}(\sigma)}{N}-\frac{{V}_{N_1,K}^1(\tau)}{N}-\frac{{V}_{N_2,K}^2(\rho)}{N}\Bigr|,\\
		\Lambda_N^1&=\max_{\tau}\Bigl|\frac{{V}_{N_1,K}^1(\tau)}{N_1}-\frac{{H}_{N_1,K}^1(\tau)}{N_1}\Bigr|,\\
        \Lambda_N^2&=\max_{\rho}\Bigl|\frac{{V}_{N_2,K}^2(\rho)}{N_2}-\frac{{H}_{N_2,K}^2(\rho)}{N_2}\Bigr|.
	\end{align*}
	Using Lemma \ref{main:proof:lem2}, 
	\begin{align*}
		\e \Lambda_N
		&\leq \frac{1}{N}\e\max_{x\in \Sigma_{2K}}|E_N^1(x_0,\ldots,x_{2K-1})|\\
		&+ \frac{1}{N}\e\max_{x\in \Sigma_{2K}}|E_N^2(x_0,\ldots,x_{2K-1})|\leq \frac{2K\sqrt{2\ln 2}}{N}
	\end{align*}
	and
	\begin{align*}
		\e\Lambda_N^\ell&= 
		\frac{1}{N_\ell}\e\max_{x\in \Sigma_{2K}}|\tilde {E}_N^\ell(x_0,\ldots,x_{2K-1})|\leq\frac{K\sqrt{2\ln 2}}{N_\ell},\,\,\ell=1,2.
	\end{align*}
	Note that similar to Lemma \ref{lem3}, the extremal processes here are all concentrated with a Gaussian-tailed bound:
	\begin{align*}
		\p(|\Lambda_N-\e \Lambda_N|\geq t)&\leq 2e^{-\frac{N^2t^2}{8K}},\\
		\p(|\Lambda_N^\ell-\e\Lambda_N^\ell|\geq t)&\leq 2e^{-\frac{N_\ell^2t^2}{4K}},\,\,\ell=1,2.
	\end{align*}
	It follows that with probability at least 
	\begin{align}\label{bridge:eq8}
		1-2e^{-N^2t^2/8K}-2e^{-N_1^2t^2/4K}-2e^{-N_2^2t^2/4K},
	\end{align}
	the following inequalities hold simultaneously, 
	\begin{align}
		\label{bridge:eq3}   \Lambda_N&\leq t+\frac{2K\sqrt{2\ln 2}}{N},\\
		\label{bridge:eq4} \Lambda_N^\ell&\leq t+\frac{K\sqrt{2\ln 2}} {N_\ell},\,\,\ell=1,2.
	\end{align}
	Next, note that for $\ell=1,2$,
	$0<\alpha\leq \alpha N/N_\ell$. 
	Since the expected maximal fitness in the $\NK$ model is nondecreasing in the parameter $K$, by Lemma \ref{monotonicity}, 
	\begin{align*}
		\e {M}_{N_\ell,K_\ell}^\ell\leq \e {M}_{N_\ell,K}^\ell\leq \sqrt{2\ln 2},
	\end{align*}
	where $K_\ell:=\lfloor \alpha (N_\ell-1)\rfloor$, and $M_{N_\ell,K_\ell}^\ell$ and $M_{N_\ell,K}^\ell$ are the optimal fitnesses of the $\NK$ models associated with the parameter pairs $(N_\ell,K_\ell)$ and $(N_\ell,K)$, respectively. Since $\e M_{N_\ell,K_\ell}^\ell\to \sqrt{2\ln 2}$ by Theorem \ref{thm:Fnk-Mnk},
	there exists some $n_0\geq 1$ depending on $\eta$ and $\alpha$ such that whenever $N\geq n_0$ and $N_1,N_2$ satisfy the second inequality in \eqref{bridge:condition}, we have
	\begin{align}\label{bridge:eq6}
		\sqrt{2\ln 2}-\eta\leq \e M_{N_\ell,K}^\ell\leq\sqrt{2\ln 2},
	\end{align}
	which combining with Lemma \ref{lem3} implies that
	with probability at least 
	\begin{align}
		\label{bridge:eq9}
		1-2e^{-Nt^2/4}-2e^{-N_1t^2/4}-2e^{-N_2t^2/4},
	\end{align}
	the following inequality holds
	\begin{align}\label{bridge:eq7}
		\sqrt{2\ln 2}-\eta-t\leq M_{N,K},{M}_{N_1,K}^1,{M}_{N_2,K}^2\leq\sqrt{2\ln 2}+t.
	\end{align}

	To establish our proof, let $E_N$ be the event on which \eqref{bridge:eq3}, \eqref{bridge:eq4}, and \eqref{bridge:eq7} hold simultaneously. Assume that on the event $E_N$, the inequality \eqref{bridge:lem1:eq1} is violated, say,
	\begin{align*}
		\frac{V_{N_1,K}^1(\tau_*)}{N_1}\leq M_{N,K}-\frac{4\eta}{c},
	\end{align*}
	then from \eqref{bridge:eq4} and \eqref{bridge:eq7},
	\begin{align*}
		\frac{{H}_{N_1,K}^1(\tau_*)}{N_1}&\leq \sqrt{2\ln 2}-\frac{4\eta}{c}+2t+\frac{K\sqrt{2\ln 2}}{N_1}.
	\end{align*}
	On the other hand, from \eqref{bridge:ass1} and \eqref{bridge:eq7}, we have
	\begin{align*}
		\frac{H_{N,K}(\sigma_*)}{N}&\geq \sqrt{2\ln 2}-2\eta-t,  \quad 
		\frac{H_{N_2,K}^2(\rho_*)}{N_2}\leq \sqrt{2\ln 2}+t.
	\end{align*}
	These together with \eqref{bridge:eq3} and \eqref{bridge:eq4} imply that
	\begin{align*}
		\sqrt{2\ln 2} -2\eta-t&\leq \frac{H_{N,K}(\sigma_*)}{N}\\
		&\leq \frac{N_1}{N}\frac{{H}_{N_1,K}^1(\tau_*)}{N_1}+\frac{N_2}{N}\frac{{H}_{N_2,K}^2(\rho_*)}{N_2}+3t+\frac{4K\sqrt{2\ln 2}}{N}\\
		&\leq \frac{N_1}{N}\Bigl(\sqrt{2\ln 2}-\frac{4\eta
		}{c}+2t+\frac{K\sqrt{2\ln 2}}{N_1}\Bigr)\\
		&\qquad +\frac{N_2}{N}\bigl(\sqrt{2\ln 2}+t\bigr)+3t+\frac{4K\sqrt{2\ln 2}}{N}\\
		&\leq \sqrt{2\ln 2}-4\eta+6t+\frac{5K\sqrt{2\ln 2}}{N}\\
		&\leq \sqrt{2\ln 2}-4\eta +6t+5\sqrt{2\ln 2}\alpha,
	\end{align*}
	where the 4th inequality  used the condition \eqref{bridge:condition}.
	Due to \eqref{bridge:eq14}, if we take $t=\eta/7$, this inequality can not  hold.
	Our proof is then completed by noting that from \eqref{bridge:eq8} and \eqref{bridge:eq9}, the probability of $E_N$ is at least $
	1-\omega e^{-Nc^2\eta^2/\omega}$ for some absolute constant $\omega>0$.
    
\medskip

    {\noindent \bf Case II: Either $N_1< K+1$ or $N_2< K+1$.} Assume that $N_1<K+1$, then $V_{N_1,K}^1(\tau^*)=0$ and $c$ is at most $N_1/N\le\alpha$. Due to Lemma \ref{lem3} and Theorem \ref{thm:Fnk-Mnk}, setting $t=\sqrt{2\ln 2}$, as long as $N$ is large enough, we have with probability at least $1-2e^{-Nt^2/4}$, 
    $$M_{N,K}-\frac{4\eta}{c}\leq \sqrt{2\ln 2}+2t-\frac{4\eta}{\alpha}<3\sqrt{2\ln 2}-4 \cdot 5\sqrt{2\ln 2} 
    =-17\sqrt{2\ln 2}.$$
    On the other hand, since $N_2=N-N_1>K$, from \eqref{bridge:eq4} with $\ell=2$ and \eqref{bridge:eq7}, letting $t=\sqrt{2\ln 2}$, as long as $N$ is large enough, with probability at least $1-2e^{-N_2t^2/4}-2e^{-N_2^2t^2/4K},$
    \begin{align*}
        \frac{V_{N_2,K}(\rho^*)}{N_2}&\geq \sqrt{2\ln 2}-2t-\frac{K\sqrt{2\ln 2}}{N_2}\\
        &\geq \sqrt{2\ln 2}-3t-\frac{\alpha \sqrt{2\ln 2}}{1-\alpha}\geq -3\sqrt{2\ln 2}.
    \end{align*}
    Putting these together yields that on the event where the above inequalities hold,
    \begin{align*}
\min\Bigl(\frac{V_{N_1,K}^1(\tau^*)}{N_1},\frac{V_{N_2,K}^2(\rho^*)}{N_2}\Bigr)\geq -3\sqrt{2\ln 2}>-17\sqrt{2\ln 2}\geq M_{N,K}-\frac{4\eta}{c}.
    \end{align*}
\end{proof}

\begin{proof}[\bf Proof of Theorem \ref{thm:near-fittest-path}]
	Fix $0<\eta<1$ and $0<\alpha<1$ that satisfy \eqref{bridge:eq14}. Let $n\geq 10$ be fixed. Denote $k=\lfloor N/(n+1)\rfloor.$ Recall the genomes $\sigma^{(0)},\sigma^{(1)},\ldots,\sigma^{(n)}$ defined in \eqref{def:bridge}.
	Let $I_0,\ldots,I_{n-1}$ be a partition of $\{0,1,2,\ldots,N-1\}$ with
	$$
	I_l=\{l k,\,lk+1,\,\ldots,\,lk+k-1\}
	$$
	for $l=0,\ldots,n-2$
	and $$I_{n-1}=\{0,1,\ldots,N-1\}\setminus\bigl(\cup_{l=0}^{n-2}I_l\bigr).$$ 
    Note that $|I_l|=k$ for $0\leq l\leq n-2$ and $k\leq |I_{n-1}|<2k.$ For $0\leq l\leq n-1$, write $I_{\leq l}=\cup_{r=0}^lI_r$ and for $0\leq l\leq n-2$, write $I_{>l}=\cup_{r=l+1}^{n-1}I_r.$  
	By definition, for any $0\leq l\leq n-1,$ $\sigma^{(l+1)}$ and $\sigma^{(l)}$ could differ only at the loci in the set $I_l$, which readily yields that
	\begin{align*}
		Q(\sigma^{(l+1)},\sigma^{(l)})&\geq \frac{N-(|I_l|+2K)}{N} \geq 1-\frac{2}{n+1}-2\alpha
	\end{align*}
	and
	\begin{align*}
		R(\sigma^{(l+1)},\sigma^{(l)})&\geq\frac{N-2|I_l|}{N}\geq 1-\frac{4}{n+1}.
	\end{align*}
	
	Next, we show that if $\hat\sigma$ and $\check\sigma$ are near-fittest genomes satisfying \eqref{bridge:eq13}, then the $\sigma^{(l)}$'s satisfy \eqref{thm:near-fittest-path:eq11}. Fix $1\leq l\leq n-1.$ Let $N_1=|I_{\leq l-1}|=kl$ and $N_2=|I_{>l}|=N-N_1.$ Note that  for $N$ large,
	\begin{align*}
		\min\Bigl(\frac{N_1}{N},\frac{N_2}{N}\Bigr)\geq \frac{k}{N}\geq \frac{1}{N}\Bigl(\frac{N}{n+1}-1\Bigr)=\frac{1}{n+1}-\frac{1}{N}\geq \frac{1}{2(n+1)}=:c.
	\end{align*}
	Recall $V_{N_1,K}^1$ and $V_{N_2,K}^2$ from \eqref{bridge:eq11}. Similarly to \eqref{bridge:eq2}, write $\hat\sigma=(\hat\tau^{(l)},\hat \rho^{(l)})$ and $\check\sigma=(\check\tau^{(l)},\check\rho^{(l)})$.  From Lemma \ref{bridge:lem1},  there exists some $N_{0}$ depending only on $\alpha,\eta$ such that for any $N\geq N_{0}$, with probability at least $1-2\omega e^{-\eta^2c^2N/\omega}$, if both $\hat \sigma$ and $\check\sigma$ have fitnesses at least $M_{N,K}-\eta$, 
	then
	\begin{align*}
		\min\Bigl(\frac{V_{N_1,K}^1(\hat\tau^{(l)})}{N_1},\frac{V_{N_2,K}^2(\hat\rho^{(l)})}{N_2}\Bigr)\geq M_{N,K}-\frac{4\eta}{c},\\
		\min\Bigl(\frac{V_{N_1,K}^1(\check\tau^{(l)})}{N_1},\frac{V_{N_2,K}^2(\check\rho^{(l)})}{N_2}\Bigr)\geq M_{N,K}-\frac{4\eta}{c}.
	\end{align*}
	Let $E_{N,l}$ be the event on which this statement holds.
	Note that $\sigma^{(\ell)}=(\check \rho^{(l)},\hat \tau^{(l)}).$ From the same argument as \eqref{bridge:eq3}, it can be shown that with probability at least $1-2e^{-N^2\eta^2/8K}\geq 1-2e^{-N\eta^2/8},$
	\begin{align*}
		\frac{1}{N}\Bigl|H_{N,K}(\sigma^{(l)})-V_{N_1,K}^1(\check\tau^{(l)})-V_{N_2,K}^2(\hat\rho^{(l)})\Big|\leq \eta+\frac{2K\sqrt{2\ln 2}}{N}\leq \eta+2\alpha \sqrt{2\ln 2}.
	\end{align*}
	Denote by $\overlineit{E}_{N,l}$ the event on which this inequality holds.
	It follows that on $E_{N,l}\cap \overlineit{E}_{N,l},$ if \eqref{bridge:eq13} holds, then
	\begin{align}
		\nonumber  \frac{H_{N,K}(\sigma^{(l)})}{N}&\geq \frac{V_{N_1,K}^1(\check\tau^{(l)})}{N}+\frac{V_{N_2,K}^2(\hat\rho^{(l)})}{N}-\eta-2\alpha \sqrt{2\ln 2}\\
		\nonumber  &\geq M_{N,K}-\frac{4\eta}{c}-\eta-2\alpha\sqrt{2\ln 2}\\
		\label{bridge:eq100} &\geq M_{N,K}-(8n+10)\eta,
	\end{align}
	where we used $c=(2(n+1))^{-1}$ and \eqref{bridge:eq14}.
	Finally, let $E_{N}=\cap_{l=0}^{n-2}(E_{N,l}\cap \overlineit{E}_{N,l})$. Then as long as $N\geq N_{0}$, 
	\begin{align*}
		\p(E_N)&\geq 1-2n\omega e^{-\eta^2c^2N/\omega}-2ne^{-N\eta^2/8}
	\end{align*}
	and on $E_N,$ if \eqref{bridge:eq13} holds, then from \eqref{bridge:eq100},
	\begin{align*}
		\min_{0\leq \ell\leq n-1}\frac{H_{N,K}(\sigma^{(l)})}{N}\geq M_{N,K}-(8n+10)\eta.
	\end{align*}
\end{proof}

\medskip

{\bf \noindent Acknowledgements.} W.-K. C. thanks National Center for Theoretical Science in Taiwan for their hospitality during his visit between May 19-30, 2025, where part of this work was completed.


\footnotesize
\begin{thebibliography}{10}

\bibitem{BAFK21}
{\sc Arous, G.~B., Fyodorov, Y.~V., and Khoruzhenko, B.~A.}
\newblock Counting equilibria of large complex systems by instability index.
\newblock {\em Proceedings of the National Academy of Sciences 118}, 34 (2021),
  e2023719118.

\bibitem{Auffinger2013Complexity}
{\sc Auffinger, A., Arous, G.~B., and Černý, J.}
\newblock Random matrices and complexity of spin glasses.
\newblock {\em Communications on Pure and Applied Mathematics 66}, 2 (2013),
  165--201.

\bibitem{auffinger2018concentration}
{\sc Auffinger, A., and Chen, W.-K.}
\newblock On concentration properties of disordered {H}amiltonians.
\newblock {\em Proceedings of the American Mathematical Society 146}, 4 (2018),
  1807--1815.

\bibitem{auffinger2018energy}
{\sc Auffinger, A., and Chen, W.-K.}
\newblock On the energy landscape of spherical spin glasses.
\newblock {\em Advances in Mathematics 330\/} (2018), 553--588.

\bibitem{auffinger2023complexity}
{\sc Auffinger, A., and Zeng, Q.}
\newblock Complexity of {G}aussian random fields with isotropic increments.
\newblock {\em Communications in Mathematical Physics 402}, 1 (2023), 951--993.

\bibitem{barra2015multi}
{\sc Barra, A., Contucci, P., Mingione, E., and Tantari, D.}
\newblock Multi-species mean field spin glasses. {R}igorous results.
\newblock In {\em Annales Henri Poincar{\'e}\/} (2015), vol.~16, Springer,
  pp.~691--708.

\bibitem{bates2025balanced}
{\sc Bates, E., and Sohn, Y.}
\newblock {B}alanced multi-species spin glasses.
\newblock {\em arXiv preprint arXiv:2507.06522\/} (2025).

\bibitem{bovier2006statistical}
{\sc Bovier, A.}
\newblock {\em Statistical mechanics of disordered systems: a mathematical
  perspective}, vol.~18.
\newblock Cambridge University Press, 2006.

\bibitem{chatterjee2008chaos}
{\sc Chatterjee, S.}
\newblock Chaos, concentration, and multiple valleys.
\newblock {\em arXiv preprint arXiv:0810.4221\/} (2008).

\bibitem{chatterjee2014superconcentration}
{\sc Chatterjee, S.}
\newblock {\em Superconcentration and related topics}, vol.~15.
\newblock Springer, 2014.

\bibitem{chen2019suboptimality}
{\sc Chen, W.-K., Gamarnik, D., Panchenko, D., and Rahman, M.}
\newblock Suboptimality of local algorithms for a class of max-cut problems.
\newblock {\em The Annals of Probability 47}, 3 (2019), 1587--1618.

\bibitem{chen2018energy}
{\sc Chen, W.-K., Handschy, M., and Lerman, G.}
\newblock On the energy landscape of the mixed even $p$-spin model.
\newblock {\em Probability Theory and Related Fields 171\/} (2018), 53--95.

\bibitem{chen2018disorder}
{\sc Chen, W.-K., and Panchenko, D.}
\newblock Disorder chaos in some diluted spin glass models.
\newblock {\em The Annals of Applied Probability 28}, 3 (2018), 1356--1378.

\bibitem{de2014empirical}
{\sc De~Visser, J. A.~G., and Krug, J.}
\newblock Empirical fitness landscapes and the predictability of evolution.
\newblock {\em Nature Reviews Genetics 15}, 7 (2014), 480--490.

\bibitem{Derrida-REM80}
{\sc Derrida, B.}
\newblock Random-{E}nergy {M}odel: Limit of a family of disordered models.
\newblock {\em Physical Review Letters 45\/} (Jul 1980), 79--82.

\bibitem{Derrida-REM81}
{\sc Derrida, B.}
\newblock Random-{E}nergy {M}odel: An exactly solvable model of disordered
  systems.
\newblock {\em Physical Review B 24\/} (Sep 1981), 2613--2626.

\bibitem{DL03}
{\sc Durrett, R., and Limic, V.}
\newblock {Rigorous Results for the $\NK$ Model}.
\newblock {\em The Annals of Probability 31}, 4 (2003), 1713--1753.

\bibitem{ES02}
{\sc Evans, S., and Steinsaltz, D.}
\newblock Estimating some features of {$\NK$} fitness landscapes.
\newblock {\em The Annals of Applied Probability 12}, 4 (2002), 1299--1321.

\bibitem{FL92}
{\sc Flyvbjerg, S.~A., and Lautrup, B.}
\newblock Evolution in rugged fitness landscapes.
\newblock {\em Physical Review A 46}, 6714--6723 (1992).

\bibitem{gamarnik2021overlap}
{\sc Gamarnik, D., and Jagannath, A.}
\newblock The overlap gap property and approximate message passing algorithms
  for $p$-spin models.
\newblock {\em The Annals of Probability 49}, 1 (2021), 180--205.

\bibitem{hashorva2005asymptotics}
{\sc Hashorva, E.}
\newblock Asymptotics and bounds for multivariate {G}aussian tails.
\newblock {\em Journal of theoretical probability 18}, 1 (2005), 79--97.

\bibitem{huang2025tight}
{\sc Huang, B., and Sellke, M.}
\newblock Tight lipschitz hardness for optimizing mean field spin glasses.
\newblock {\em Communications on Pure and Applied Mathematics 78}, 1 (2025),
  60--119.

\bibitem{Huang2021}
{\sc Huang, H.}
\newblock {\em Random Energy Model}.
\newblock Springer Nature Singapore, Singapore, 2021, pp.~59--62.

\bibitem{Kauffman93}
{\sc Kauffman, S.}
\newblock {\em {The Origins of Order: Self-Organization and Selection in
  Evolution.}}
\newblock OxfordUniv.Press, 1993.

\bibitem{KauffmanLevin198711}
{\sc Kauffman, S., and Levin, S.}
\newblock Towards a general theory of adaptive walks on rugged landscapes.
\newblock {\em Journal of Theoretical Biology 128}, 1 (1987), 11--45.

\bibitem{Kauffman1989}
{\sc Kauffman, S.~A., and Weinberger, E.~D.}
\newblock The {$\NK$} model of rugged fitness landscapes and its application to
  maturation of the immune response.
\newblock {\em Journal of Theoretical Biology 141}, 2 (1989), 211--245.

\bibitem{li2002normal}
{\sc Li, W., and Shao, Q.-M.}
\newblock A normal comparison inequality and its applications.
\newblock {\em Probability Theory and Related Fields 122}, 4 (2002), 494--508.

\bibitem{LP04}
{\sc Limic, V., and Pemantle, R.}
\newblock {More rigorous results on the Kauffman--Levin model of evolution}.
\newblock {\em The Annals of Probability 32}, 3 (2004), 2149--2178.

\bibitem{Lush1935}
{\sc Lush, J.~L.}
\newblock Progeny test and individual performance as indicators of an animal's
  breeding value.
\newblock {\em Journal of Dairy Science 18}, 1 (1935), 1--19.

\bibitem{MP89}
{\sc Macken, C.~A., and Perelson, A.~S.}
\newblock Protein evolution on rugged landscapes.
\newblock {\em Proceedings of the National Academy of Sciences 86}, 6191-6195 (1989).

\bibitem{MM2009-chp5}
{\sc Mézard, M., and Montanari, A.}
\newblock {\em The {R}andom {E}nergy {M}odel}.
\newblock Oxford University Press, 2009, ch.~5, pp.~93--106.

\bibitem{Panchenko2008}
{\sc Panchenko, D.}
\newblock On differentiability of the {P}arisi formula.
\newblock {\em Electronic Communications in Probability 13\/} (2008), 241--247.

\bibitem{panchenkobook}
{\sc Panchenko, D.}
\newblock The {S}herrington-{K}irkpatrick model.
\newblock {\em Springer Monographs in Mathematics\/} (2013).

\bibitem{panchenko2014parisi}
{\sc Panchenko, D.}
\newblock The {P}arisi formula for mixed $p$-spin models.
\newblock {\em The Annals of Probability 42}, 3 (2014), 946--958.

\bibitem{panchenko2015free}
{\sc Panchenko, D.}
\newblock The free energy in a multi-species {S}herrington--{K}irkpatrick
  model.
\newblock {\em The Annals of Probability 43}, 6 (2015), 3494--3513.

\bibitem{papkou2023rugged}
{\sc Papkou, A., Garcia-Pastor, L., Escudero, J., and Wagner, A.}
\newblock A rugged yet easily navigable fitness landscape.
\newblock {\em Science 382}, 6673 (2023), eadh3860.

\bibitem{reetz2008constructing}
{\sc Reetz, M.~T., and Sanchis, J.}
\newblock Constructing and analyzing the fitness landscape of an experimental
  evolutionary process.
\newblock {\em Chembiochem : a European journal of chemical biology 9}, 14 (2008), 2260--2267.

\bibitem{Resnick2007extreme}
{\sc Resnick, S.~I.}
\newblock {\em {Extreme Values, Regular Variation, and Point Processes}}.
\newblock Springer Science {\&} Business Media, 2007.

\bibitem{robustgeneticcode24}
{\sc Rozhoňová, H., Martí-Gómez, C., McCandlish, D.~M., and Payne, J.~L.}
\newblock Robust genetic codes enhance protein evolvability.
\newblock {\em PLOS Biology 22}, 5 (2024), 1--28.

\bibitem{proteinevo-comp25}
{\sc Sandhu, M., Chen, J.~Z., Matthews, D.~S., Spence, M.~A., Pulsford, S.~B.,
  Gall, B., Kaczmarski, J.~A., Nichols, J., Tokuriki, N., and Jackson, C.~J.}
\newblock Computational and experimental exploration of protein fitness
  landscapes: Navigating smooth and rugged terrains.
\newblock {\em Biochemistry 64}, 8 (2025), 1673--1684.

\bibitem{fitness-evo02}
{\sc Smith, T., Husbands, P., Layzell, P., and O'Shea, M.}
\newblock Fitness landscapes and evolvability.
\newblock {\em Evolutionary Computation 10}, 1 (2025), 1--34.

\bibitem{Stein1992}
{\sc Stein, D.~L.}
\newblock {\em Spin Glasses and Biology}.
\newblock World Scientific, 1992.

\bibitem{talagrand2010mean}
{\sc Talagrand, M.}
\newblock {\em Mean field models for spin glasses: {V}olume {I}: {B}asic
  examples}, vol.~54.
\newblock Springer Science \& Business Media, 2010.

\bibitem{talagrand2011mean}
{\sc Talagrand, M.}
\newblock Mean field models for spin glasses: Advanced replica-symmetry and low
  temperature. {V}olume {II}.
\newblock {\em Ergebnisse der Mathematik und Ihrer Grenzgebiete 3\/} (2011).

\bibitem{TopoComplexPRL13}
{\sc Wainrib, G., and Touboul, J.}
\newblock Topological and dynamical complexity of random neural networks.
\newblock {\em Physical Review Letters 110\/} (2013), 118101.

\bibitem{Weinberger89}
{\sc Weinberger, E.~D.}
\newblock A more rigorous derivation of some properties of uncorrelated fitness
  landscapes.
\newblock {\em Journal of Theoretical Biology134\/} (1989), 125--129.

\bibitem{Weinberger91}
{\sc Weinberger, E.~D.}
\newblock Local properties of {K}auffman's {N-K} model: {A} tunably rugged
  energy landscape.
\newblock {\em Physical Review A 44\/} (1991), 6399--6413.

\bibitem{WTZ2000}
{\sc Wright, A.~H., Thompson, R.~K., and Zhang, J.}
\newblock {The Computational Complexity of $NK$ Fitness Functions}.
\newblock {\em IEEE Transactions on Evolutionary Computation 4}, 4 (2000), 373--379.

\bibitem{Wright1932}
{\sc Wright, S.}
\newblock The roles of mutation, inbreeding, crossbreeding and selection in
  evolution.
\newblock {\em Proceedings of the XI International Congress of Genetics 8\/}
  (1932), 209--222.

\end{thebibliography}
\end{document}